\documentclass[a4paper]{article}
\usepackage{amsmath}
\usepackage{amssymb}
\usepackage{amsthm}
\usepackage[english]{babel}
\usepackage[utf8]{inputenc}
\usepackage[T1]{fontenc}
\usepackage{microtype}
\usepackage{enumerate}
\usepackage{mathrsfs}
\usepackage{mathtools}
\usepackage{graphicx}
\usepackage{float}
\usepackage{dsfont}
\usepackage[colorinlistoftodos,prependcaption]{todonotes}
\usepackage{listings}
\usepackage{booktabs}
\usepackage{mathtools}
\usepackage{adjustbox}
\usepackage{wrapfig}
\usepackage{xargs}
\usepackage{xcolor}
\usepackage{csquotes}
\usepackage{cite}

\theoremstyle{plain}
\newtheorem{teo}{Theorem}[section]
\newtheorem{lem}[teo]{Lemma}
\newtheorem{cor}[teo]{Corollary}
\newtheorem{prop}[teo]{Proposition}
\newtheorem*{teo*}{Theorem}

\theoremstyle{remark}
\newtheorem{oss}{Remark}[section]

\theoremstyle{definition}
\newtheorem{defi}[teo]{Definition}

\newtheorem*{defi*}{Definition}

\renewcommand{\phi}{\varphi}

\newcommandx{\checkk}[2][1=]{\todo[linecolor=red,backgroundcolor=red!25,bordercolor=red,#1]{#2}}
\newcommandx{\fix}[2][1=]{\todo[linecolor=blue,backgroundcolor=blue!25,bordercolor=blue,#1]{#2}}
\newcommandx{\improve}[2][1=]{\todo[linecolor=green,backgroundcolor=green!25,bordercolor=green,#1]{#2}}

\nocite{*}

\date{\today}
\title{A particle system approach to aggregation phenomena}
\author{Franco Flandoli\footnote{franco.flandoli@sns.it. Scuola Normale Superiore of Pisa, Italy.}, Marta Leocata\footnote{leocata@mail.dm.unipi.it. University of Pisa, Italy.}}

\begin{document}
\maketitle 
\begin{abstract}
Inspired by a PDE-ODE system of aggregation developed in the biomathematical
literature, an interacting particle system representing aggregation at the
level of individuals is investigated. It is proved that the empirical density
of the individual converges to solution of the PDE-ODE system.
\end{abstract}
\section{Introduction}

The mathematical literature applied to Biology and Social Science is rich of
models devoted to the description of aggregation. Motivations come from
several problems like embryo development, tissue homeostasis, tumor growth,
animal swarming and flocking. The literature presents heterogeneous
mathematical tools: discrete and continuous individual based model, ordinary
and partial differential equations (resp. ODE and PDE) and mixture of the
previous ones. Also because of this heterogeneity, an interesting issue is to
justify the PDE models through the investigation of scaling limits of models
based on interaction between individuals. Following this general program, in
this work we propose an individual based model and prove convergence, when the
number of individuals goes to infinity, to a class of PDE-ODE systems which
includes the so called Armstrong-Painter-Sherratt model proposed in
\cite{ArmPainShe}, \cite{PainArmShe}, evolution of a previous model of
\cite{Perumpani}, including in particular a form of delay by coupling the system
with an ODE.

We assume that individuals interact between each other by looking at the
density field produced by the others: think for instance to the motion of
animals in a swarm or a flock; presumably each animal moves driven by a
general overview of the others, not computing several pairwise interactions.
Let $N$ be the number of individuals and $X_{t}^{i,N}$, $i=1,...,N$, be their
positions. We model this particle-density interaction by the following
equations:%
\begin{equation}
dX_{t}^{i,N}=\int_{\mathbb{R}^{d}}\frac{y-X_{t}^{i,N}}{|y-X_{t}^{i,N}%
|}g(|y-X_{t}^{i,N}|,u_{t}^{N}(y),m_{t}^{N}(y))dydt+\sqrt{2}dB_{t}%
^{i}\label{micro_cells}%
\end{equation}
\begin{equation}
\frac{\partial m_{t}^{N}\left(  x\right)  }{\partial t}=-\lambda u_{t}%
^{N}(x)(m_{t}^{N}\left(  x\right)  )^{\zeta},\qquad x\in\mathbb{R}%
^{d}\label{micro_ecm}%
\end{equation}
for $i=1,...,N$ and $t\in\left[  0,T\right]  $, where $u_{t}^{N}(y)$ is a
density associated to the population of particles, defined below, $m_{t}%
^{N}\left(  x\right)  $ is the field which allows a dependence on the past or
may be used to model external effects like those of the Extracellular Matrix,
$\zeta$ is typically equal to 1 or 2, and $B_{t}^{i}$ are independent Brownian
motions $B_{t}^{i}$ accounting for a random component of the motion. Each
particle $X_{t}^{i,N}$ interacts with each location $y$; the direction of the
force is given by the unitary vector $\frac{y-X_{t}^{i,N}}{|y-X_{t}^{i,N}|}$
which spans the line between the particles; the strength of the interaction is
$g(|y-X_{t}^{i,N}|,u_{t}^{N}(y))$, namely it is modulated by the distance
$|y-X_{t}^{i,N}|$, by the density $u_{t}^{N}(y)$ and by the external field
$m_{t}^{N}(y)$. At positions $y$ where $g>0$, particle $X_{t}^{i,N}$ moves
towards $y$, namely have a tendency to aggregate. Using different functions
$g$ we may describe different kinds of attraction;a wide discussion is
presented in the last section of the paper. A technical issue concerns the
definition of the density $u_{t}^{N}(x)$, see the discussion below. Under
suitable assumptions, our main theorem is the convergence of the previous
particle model to the PDE-ODE system
\begin{align}
&  \frac{\partial u_{t}}{\partial t}=\Delta u_{t}-\text{div}(u_{t}%
b(u_{t},m_{t}))\nonumber\\
&  \frac{\partial m_{t}}{\partial t}=-\lambda u_{t}m_{t}^{\zeta}\label{eq:21}%
\end{align}
on $[0,T]\times\mathbb{R}^{d}$, where%
\begin{equation}
b(u,m)(x):=\int_{\mathbb{R}^{d}}\frac{y-x}{|y-x|}g(|y-x|,u(y),m\left(
y\right)  )dy.\label{def b}%
\end{equation}

Let us finally discuss the concept of density $u_{t}^{N}(x)$. Given the
particles $X_{t}^{i,N}$, one first associates to them the classical concept of
empirical measure:
\[
S_{t}^{N}\left(  dx\right)  :=\frac{1}{N}\sum_{i=1}^{N}\delta_{X_{t}^{i,N}%
}\left(  dx\right)  .
\]
Its direct use, however, in the previous modelling would oblige us to choose
functions $g$ depending on measures, instead of functions, which are less easy
to formulate in examples. And, more importantly, we could not speak of
$u_{t}^{N}\left(  y\right)  $, the density at position $y$. In numerics it is
common to overcome this difficulty by the so called \textit{kernel smoothing},
which consists in mollifying the measure by convolution with a smooth kernel.
We adopt this procedure. We choose a smooth, compactly supported, probability
density $W$ (the kernel) and rescale it with $N$ in a suitable way. A general
form of rescaling is%
\[
W_{N}(x):=N^{\beta}W(N^{\beta/d}x)
\]
for some $\beta\in\left(  0,1\right)  $, as suggested by K. Oelschlager
\cite{Oelsch}. The density $u_{t}^{N}(x)$ is thus given by%
\[
u_{t}^{N}(x):=\left(  W_{N}\ast S_{t}^{N}\right)  \left(  x\right)
=\sum_{i=1}^{N}W_{N}(x-X^{i,N}).
\]
Thanks to the semigroup approach that we implement in the estimates on the
particle system, we are able to consider any choice of $\beta\in\left(
0,1\right)  $. This is not a trivial task, since other approaches require more
restrictions on $\beta$, see \cite{Oelsch}, \cite{TrevisanNeklydov}.

The paper is structured as follows: in Section 2 we give some notations,
formulate the main result and prove some preliminary facts; in Section 3 we
prove tightness of the density $u_{t}^{N}(x)$ in suitable spaces; in Section 4
we show the passage to the limit and complete the proof of the main result;
finally in Section 5 we discuss several examples of interaction function $g$
and show by numerical simulations that the previous model may catch different
kinds of aggregation pattern.

\section{Notations and basic results}

\subsection{The particle system}

For every positive integer $N$, we consider a particle system described by
equations (\ref{micro_cells}) coupled with the random field $m_{t}^{N}\left(
x\right)  $ satisfying (\ref{micro_ecm}) for some integer $\zeta\geq1$, with
initial conditions $X_{0}^{i,N}=X_{0}^{i}$, $i=1,...,N$, where $B_{t}^{i}$,
$i\in\mathbb{N}$, is a sequence of independent Brownian motions on a filtered
probability space $\left(  \Omega,\mathcal{F},\mathcal{F}_{t},P\right)  $;
$X_{0}^{i}$, $i\in\mathbb{N}$, is a sequence of $\mathcal{F}_{0}$-measurable
independent random variables with values in $\mathbb{R}^{d}$, identically
distributed with density $u_{0}$; the random function $u_{t}^{N}$ is given by
$u_{t}^{N}(x):=\left(  W_{N}\ast S_{t}^{N}\right)  \left(  x\right)  $ where
$S_{t}^{N}=\frac{1}{N}\sum_{i=1}^{N}\delta_{X_{t}^{i,N}}$ and $W_{N}%
(x):=N^{\beta}W(N^{\beta/d}x)$ for some $\beta\in\left(  0,1\right)  $; the
random fields $m_{t}^{N}$ have initial conditions $m_{0}^{N}\left(  x\right)
=m_{0}\left(  x\right)  $ where $m_{0}:\mathbb{R}^{d}\rightarrow\mathbb{R}$ is
a measurable function with $0\leq m_{0}\leq M$, and the functional
\[
b:L^{2}(\mathbb{R}^{d})\times L^{2}(\mathbb{R}^{d})\rightarrow L^{\infty
}(\mathbb{R}^{d})
\]
is given by (\ref{def b}) where $g:\mathbb{R}^{+}\times\mathbb{R}^{+}%
\times\mathbb{R}^{+}\rightarrow\mathbb{R}$, $g=g(r,u,m)$, is differentiable,
bounded with bounded derivatives, and satisfies%
\begin{equation}\label{hyp:g}
|g(r,u,m)|+|\nabla g(r,u,m)|\leq C\cdot\exp(-r)
\end{equation}
for some constant $C>0$ (where $\nabla g$ denotes the gradient in all
variables). It follows that, for every pair of measurable functions
$u(x),m(x)$, the aggregation force is bounded:
\begin{equation}
|b(u,m)(x)|\leq\int_{\mathbb{R}^{d}}\left\vert g(|y-x|,u(y),m(y))\right\vert
dy\leq C\int_{\mathbb{R}^{d}}e^{-|x-y|}dy:=C^{\prime}<\infty.\label{bounded b}%
\end{equation}
We also have%
\begin{equation}
|b(u,m)(x)-b(u^{\prime},m^{\prime})(x)|\leq C\cdot\int_{\mathbb{R}^{d}%
}e^{-|x-y|}\left(  \left\vert u\left(  y\right)  -u^{\prime}\left(  y\right)
\right\vert +\left\vert m\left(  y\right)  -m^{\prime}\left(  y\right)
\right\vert \right)  dy\label{lip 1 b}%
\end{equation}
and regarding the derivative, due to the condition on the gradient of $g$:
\begin{multline}
\label{lip 2 b} 
|\nabla_{x}\cdot b(u,m)(x)|\leq\left| \int_{\mathbb{R}^{d}}\nabla_x\cdot\left(\frac{y-x}{|y-x|} g(|y-x|,u(y),m(y))\right)dy\right|\leq\\
|g(0,u(x),m(x))|+\left|\int_{\mathbb{R}^{d}}\frac{y-x}{|y-x|}\cdot\nabla_x\left(g(|y-x|,u(y),m(y))\right)dy\right|\leq\\
C_1+\int_{\mathbb{R}^d}|\partial_r g(|y-x|,u(y),m(y))|dy\leq C_1+C_2.
\end{multline}
Under these assuptions, existence and uniqueness of a solution, for finite
$N$, of system (\ref{micro_cells})-(\ref{micro_ecm}) can be proved by
classical methods. Let us explain some details. Let us denote by $C\left(
L^{2}\right)  $, $C_{+}\left(  L^{2}\right)  $, $C_{0,M}\left(  L^{2}\right)$,  the spaces 
%
%
%
\[
C\left(  L^{2}\right)     :=C([0,T],L^{2}(\mathbb{R}^{d}))\]
\[
C_{+}\left(  L^{2}\right)    :=\left\{  u\in C\left(  L^{2}\right)
:u_{t}\geq0\text{ for all }t\in\left[  0,T\right]  \right\} \]
\[
C_{0,M}\left(  L^{2}\right)     =\left\{  m\in C\left(  L^{2}\right)  :0\leq
m_{t}\leq M\text{ for all }t\in\left[  0,T\right]  \right\} \]

We say that a random field $m_{t}^{N}\left(  x\right)  $, $t\in\left[
0,T\right]  $, $x\in\mathbb{R}^{d}$ defined on $\left(  \Omega,\mathcal{F}%
,\mathcal{F}_{t},P\right)  $, is adapted of class $C_{0,M}\left(
L^{2}\right)  $ if $P$-a.s. the functions $\left(  t,x\right)  \mapsto
m_{t}^{N}\left(  x\right)  $ belong to $C_{0,M}\left(  L^{2}\right)  $ and for
every $t\in\left[  0,T\right]  $ the function $\left(  x,\omega\right)
\mapsto m_{t}^{N}\left(  x,\omega\right)  $ is $\mathcal{B}\left(
\mathbb{R}^{d}\right)  \times\mathcal{F}_{t}$-measurable. We say that
$\left(  X^{1,N},...,X^{N,N},m^{N}\right)  $ is a strong solution of system
(\ref{micro_cells})-(\ref{micro_ecm}) if $X_{t}^{1,N},...,X_{t}^{N,N}$ are
continuous adapted processes on $\left(  \Omega,\mathcal{F},\mathcal{F}%
_{t},P\right)  $, $m_{t}^{N}\left(  x\right)  $ is adapted of class
$C_{0,M}\left(  L^{2}\right)  $, all defined on $\left(  \Omega,\mathcal{F}%
,\mathcal{F}_{t},P\right)  $, and identities (\ref{micro_cells}%
)-(\ref{micro_ecm}) hold, with the equations understood integrated in time. We
say that pathwise uniqueness hold if two such solutions are indistinguishable processes.

\begin{prop}
\label{prop particles}Given any positive integer $N$ and any function
$m_{0}\in L^{2}(\mathbb{R}^{d})$ such that $0\leq m_{0}\leq M$, there exists a
strong solution of system (\ref{micro_cells})-(\ref{micro_ecm}) and pathwise
uniqueness hold.
\end{prop}

\begin{proof}
The proof is classical, we explain only the idea. Given an integer $\zeta
\geq1$, $u^{N}$ and a.e. $x\in\mathbb{R}^{d}$, the solution of equation
(\ref{micro_ecm}) is global, unique and explicit:%
\[
m_{t}^{N}(x)=F_{\zeta}\left(  m_{0}(x),\int_{0}^{t}u_{s}^{N}(x)ds\right)
\]%
\[
F_{\zeta}\left(  a,b\right)  =a\cdot\widetilde{F}_{\zeta}\left(  a,b\right)
\]%
\[
\widetilde{F}_{\zeta}\left(  a,b\right)  =%
\begin{cases}
\exp\left(  -\lambda b\right)  \quad\text{if }\zeta=1\\
\frac{1}{\left[  a^{\zeta-1}(\zeta-1)\lambda b+1\right]  ^{\frac{1}{\zeta-1}}%
}\quad\text{if }\zeta\geq2\text{.}%
\end{cases}
\]
The function $\widetilde{F}_{\zeta}:\left[  0,M\right]  \times\lbrack
0,\infty)\rightarrow\mathbb{R}$ is bounded and the function $F_{\zeta}:\left[
0,M\right]  \times\lbrack0,\infty)\rightarrow\mathbb{R}$ is Lipschitz
continuous, with at most linear growth in $a$, uniformly in $b$. Then one may
consider the system of integral equations%
\begin{equation}
X_{t}^{i,N}=X_{0}^{i}+\int_{0}^{t}b(u_{s}^{N},F_{\zeta}\left(  m_{0}%
(\cdot),\int_{0}^{s}u_{r}^{N}(\cdot)dr\right)  )(X_{s}^{i,N})ds+\sqrt{2}%
B_{t}^{i}\quad\text{i=1,\dots,N} \label{reduced system}%
\end{equation}
as a closed system, with only the variables $X_{t}^{1,N},...,X_{t}^{N,N}$. It
is a path-dependent equation: the past appears in the drift; but this does not
change the way contraction principle applies. One can check that strong
existence and pathwise uniqueness for the original system for the variables
$\left(  X^{1,N},...,X^{N,N},m^{N}\right)  $ is equivalent to strong existence
and pathwise uniqueness for this reduced path-dependent system in the
variables $\left(  X^{1,N},...,X^{N,N}\right)  $ only;\ property $m^{N}\in
C_{0,M}\left(  L^{2}\right)  $ is deduced from the explicit formula.
Let us
say how to prove existence and uniqueness for (\ref{reduced system}).
Thanks to the property \eqref{lip 2 b} the drift of equation
(\ref{reduced system}) is globally Lipschitz continuous. We define the family of maps $J^i$ as
\begin{multline*}
J^i:E\to \mathbb{R}\quad J^i(Y):=X_{0}^{i}+\int_{0}^{t}b(u_{s}^{N},F_{\zeta}\left(  m_{0}%
(\cdot),\int_{0}^{s}u_{r}^{N}(\cdot)dr\right)  )(Y)ds+\sqrt{2}
B_{t}^{i}\\
i=1,\dots,N
\end{multline*}
where $E=L^2_{\mathcal{F}}(\Omega, C([0,T'],\mathbb{R}^d))$, with $T'<T$. Then with classical computation we get that $J^i$ is a contraction on the space E:
\[\left|\left|J^i(Y)-J^i(Y')\right|\right|_E\leq CT'\left|\left|Y-Y'\right|\right|_E\]
choosing $CT'<1$. Hence local existence and uniqueness of strong solutions is proved. Iterating this argument one can get the global existence result, because the amplitude of the interval of iteration depends only on $CT'$, namely it is fixed for each iteration. 
\end{proof}

\begin{oss}
Existence and uniqueness of solution of the system (\ref{micro_cells})-(\ref{micro_ecm}) could be obtained following another approach. With less effort could be possible to obtain just weak existence and uniqueness in law for the system (\ref{micro_cells})-(\ref{micro_ecm}): the method of creating weak solution to SDEs is transformation of drift via Girsanov theorem, see \cite{KaratzasShreve}. 
Being the drift $b$ bounded, see condition \eqref{bounded b}, hypotheses of  Proposition 3.6 and Proposition 3.10 of \cite{KaratzasShreve} are verified and existence and uniqueness of the system is obtained.
 Then $X^{i,N}_t$ is solution of (\ref{micro_cells}). Thus also $m^N_t$ exists, is unique and explicit. This kind of existence would be enough for the purpose of the paper, but we still to decide to emphasize in Proposition \ref{prop particles} that a stronger result is attainable.
\end{oss}

\subsection{Main results\label{sect main res}}

After the indentity of Lemma \ref{lem:1} below for the empirical measure is
proved, it is natural to conjecture that the limit of the pair $\left(
u_{t}^{N}\left(  x\right)  ,m_{t}^{N}\left(  x\right)  \right)  $ solves the
system (\ref{eq:21}) with initial condition $\left(  u_{0},m_{0}\right)  $,
where $u_{0}$ is the density of the r.v.'s $X_{0}^{i}$ and $m_{0}$ is the
limit of $m_{0}^{N}$. We interpret the first equation of this system in the so
called mild form and the second one in integral form. Concerning the
initial conditions, we make a choice of simplicity. We assume that $u_{0}:\mathbb{R}^{d}%
\rightarrow\mathbb{R}$ (the initial distribution of individuals) is a
probability density of class $C^{1}$ with compact support, see Lemma \ref{lem:initial_condition}. About
$m_{0}:\mathbb{R}^{d}\rightarrow\mathbb{R}$, we assume it is of class
$L^{2}(\mathbb{R}^{d})$ and $0\leq m_{0}\leq M$.

\begin{defi}
By mild solution of system (\ref{eq:21}) we mean a pair $\left(  u,m\right)  $
belonging to $C_{+}\left(  L^{2}\right)  \times C_{0,M}\left(  L^{2}\right)  $
such that
\begin{align*}
&  u_{t}(x)=e^{t\Delta}u_{0}+\int_{0}^{t}\nabla\cdot e^{(t-s)\Delta}%
(u_{s}b(u_{s},m_{s}))ds\\
&  m_{t}(x)=m_{0}(x)-\int_{0}^{t}\lambda u_{s}(x)m_{s}^{\zeta}(x)ds.
\end{align*}

\end{defi}

Where $e^{tA}$ denote the heat semigroup, more precisely defined in Section \ref{sect semigr}. Notice that the $L^{2}(\mathbb{R}^{d})$-norm of $u_{s}b(u_{s},m_{s})$ is
bounded, since $b$ is bounded and $u\in C\left(  L^{2}\right)  $. Hence
$\nabla\cdot e^{(t-s)\Delta}(u_{s}b(u_{s},m_{s}))$ is integrable by property
(\ref{prop:2}) below. Convergence of the particles system is proved only
locally in space, hence we introduce the space%
\[
C\left(  L_{loc}^{2}\right)  :=C([0,T],L_{loc}^{2}(\mathbb{R}^{d}))
\]
where the topology on $L_{loc}^{2}(\mathbb{R}^{d})$ is given by the metric%
\[
d_{L_{loc}^{2}}\left(  f,g\right)  =\sum_{n=1}^{\infty}2^{-n}\left(
\left\Vert f-g\right\Vert _{L^{2}\left(  B\left(  0,n\right)  \right)  }%
\wedge1\right)  .
\]

\begin{teo}
\label{main thm}System (\ref{eq:21}) has one and only one mild solution
$\left(  u,m\right)  $ in $C_{+}\left(  L^{2}\right)  \times C_{0,M}\left(
L^{2}\right)  $; and the pair $\left(  u^{N},m^{N}\right)  $ converges to
$\left(  u,m\right)  $ in $C\left(  L_{loc}^{2}\right)  \times C\left(
L_{loc}^{2}\right)  $, in probability.
\end{teo}

\subsection{Some useful properties of Analytic Semigroup\label{sect semigr}}
We denote with $W^{\alpha,2}(\mathbb{R}^d)$ the fractional sobolev space, which is a Banach space with the norm 
\[
\left|\left|f\right|\right|_{W^{\alpha,2}(\mathbb{R}^d)}=\left|\left| f\right|\right|_{L^2(\mathbb{R}^d)}+[f]_{\alpha,2,\mathbb{R}^d}
\]
where $$[f]_{\alpha,2,\mathbb{R}^d}=\int_{\mathbb{R}^d}\int_{\mathbb{R}^d} \frac{|f(x)-f(y)|^2}{|x-y|^{2\alpha+d}}dxdy.$$
Or equivantely
\begin{equation*}
\label{eq:u0}
\left|\left|u^N_0\right|\right|_{W^{\alpha,2}}=\left\Vert u_{0}^{N}\right\Vert _{L^{2}\left(
\mathbb{R}^{d}\right)  } + \left\Vert \left(
-\Delta\right)  ^{\alpha}u_{0}^{N}\right\Vert _{L^{2}\left(  \mathbb{R}%
^{d}\right)  }.
\end{equation*}
where the fractional laplacian can be characterized by the following lemma.

\begin{lem}
Let $s\in\left(  0,1\right)  $. Then there exists a constant $C\left(
s,d\right)  $ such that
\[
\left(  -\Delta\right)  ^{s}f\left(  x\right)  =C\left(  s,d\right)
\int_{\mathbb{R}^{d}}\frac{f\left(  x+y\right)  +f\left(  x-y\right)
-2f\left(  x\right)  }{\left\vert y\right\vert ^{d+2s}}dy,\qquad
x\in\mathbb{R}^{d}%
\]
for every compact support twice differentiable function $f$.
\end{lem}
Notice that boundedness of $f$ guarantees integrability at infinity, while
twice differentiability implies that the numerator is, for small $\left\vert
y\right\vert $, infinitesimal of order two, which compensates the singularity
of the denominator.
Another very useful property on the fractional laplacian is that it is a local operator, namely it preserves compact support of functions. 

Let us recall some well known properties of analytical semigroups. The family
of operators
\[
\left(  e^{tA}f\right)  \left(  x\right)  :=\int_{\mathbb{R}^{d}}\frac
{1}{(2\pi \sigma^2t)^{d/2}}e^{-\frac{|x-y|^{2}}{2\sigma^2t}}f(y)dy
\]
for $t\geq0$, defines an analytic semigroup (the heat semigroup) on the space
$W^{\alpha,2}(\mathbb{R}^{d})$, for every $\alpha\geq0$. The infinitesimal
generator in $L^{2}(\mathbb{R}^{d})$ is the operator $A:D(A)\subset
L^{2}(\mathbb{R}^{d})\longrightarrow L^{2}(\mathbb{R}^{d})$, $D(A)=W^{2,2}%
(\mathbb{R}^{d})$, given by $Af=\frac{\sigma^2}{2}\Delta f$. It is possible to define fractional
power of the operator $(I-A)^{\delta}$ for $\delta\in\mathbb{R}$ and a well
known fact is the equivalence of norms:
\begin{equation}
||(I-A)^{\delta/2}f||_{L^{2}}\sim||f||_{W^{\delta,2}} \label{prop1_semigroup}%
\end{equation}
Another property, often used in the sequel, is that for every $\delta,T>0$

there is a constant $C_{\delta,T}$ such that for $t\in(0,T]$
\begin{equation}
||(I-A)^{\delta}e^{tA}||_{L^{2}\rightarrow L^{2}}\leq\frac{C_{\delta,T}%
}{t^{\delta}}. \label{prop2_semigroup}%
\end{equation}
Finally, we remark that the operator $\nabla(I-A)^{-1/2}$ is bounded in
$L^{2}$
\begin{equation}
||\nabla(I-A)^{-1/2}||_{L^{2}\rightarrow L^{2}}\leq C \label{prop3_semigroup}%
\end{equation}
where, here and below, we continue to write simply $L^{2}$ also when the
functions are vector valued, as in the case of $\nabla(I-A)^{-1/2}f$.

It will be useful to know the following result on the improvement of regularity.

\begin{lem}
\label{lemma regularity}If $u\in L^{p}\left(  0,T;L^{2}(\mathbb{R}%
^{d})\right)  $ for some $p>2$ and satisfies
\[
u_{t}(x)=e^{tA}u_{0}+\int_{0}^{t}\nabla\cdot e^{(t-s)A}(u_{s}%
b_{s})ds
\]
for some bounded measurable function $b$, then $u\in C([0,T],L^{2}%
(\mathbb{R}^{d}))$.
\end{lem}

\begin{proof}
The product $ub$ is in $L^{p}\left(  0,T;L^{2}(\mathbb{R}^{d})\right)  $.
Using the bound
\[
||\nabla\cdot e^{(t-s)A}||_{L^{2}\rightarrow L^{2}}=||\nabla
\cdot(I-A)^{-1/2}(I-A)^{1/2}e^{(t-s)A}||_{L^{2}\rightarrow L^{2}}%
\leq\frac{C}{\left(  t-s\right)  ^{1/2}}%
\]
we deduce that $t\mapsto\int_{0}^{t}\nabla\cdot e^{(t-s)A}(u_{s}b_{s})ds$
is of class $C([0,T],L^{2}(\mathbb{R}^{d}))$;\ the same is true for $t\mapsto
e^{tA}u_{0}$ because $u\in L^{2}(\mathbb{R}^{d})$ as a byproduct of our
assumptions. Hence $u\in C([0,T],L^{2}(\mathbb{R}^{d}))$.
\end{proof}

\subsection{Preliminary results}

\begin{lem}
\label{lem:1} For every $\phi\in C^{2}([0,T]\times\mathbb{R}^{d})$, $S_{t}%
^{N}$ satisfies the following identity:
\begin{align*}
\left\langle S_{t}^{N},\phi_{t}\right\rangle -\left\langle S_{0}^{N},\phi
_{0}\right\rangle  &  =\int_{0}^{t}\left\langle S_{s}^{N},\frac{\partial
\phi_{s}}{\partial s}\right\rangle ds+\frac{\sigma^2}{2}\int_{0}^{t}\left\langle S_{s}%
^{N},\Delta\phi_{s}\right\rangle ds+\\
&  +\int_{0}^{t}\left\langle S_{s}^{N},\nabla\phi_{s}\cdot b(u_{s}^{N}%
,m_{s}^{N})\right\rangle ds+M_{t}^{N,\phi}%
\end{align*}
where
\[
M_{t}^{N,\phi}=\frac{\sigma}{N}\sum_{i=1}^{N}\int_{0}^{t}\nabla\phi
(X_{s}^{i,N})\cdot dB_{s}^{i}.
\]
In particular, choosing $\phi\left(  \cdot\right)  =\phi_{x}\left(
\cdot\right)  =W_{N}(\cdot-x)$, for $x\in\mathbb{R}^{d}$, we get
\[
u_{t}^{N}(x)-u_{0}^{N}(x)=\frac{\sigma^2}{2}\int_{0}^{t}\Delta u_{s}^{N}(x)ds+\int_{0}%
^{t}\mathrm{div}(W_{N}\ast\left(  b(u_{s}^{N},m_{s}^{N})S_{s}^{N}\right)
)(x)ds+M_{t}^{N}\left(  x\right)
\]
where
\[
M_{t}^{N}\left(  x\right)  =\frac{\sigma}{N}\sum_{i=1}^{N}\int_{0}^{t}\nabla
W_{N}(x-X_{s}^{i,N})\cdot dB_{s}^{i}.
\]

\end{lem}

\begin{proof}
The proof follows by It\^{o} formula and Gauss Green formula.
\end{proof}

Concerning the family of mollifiers, we have the following useful properties,
whose proof is an elementary computation, see for instance
\cite{FlandoliLeimbachOlivera}.

\begin{lem}
\label{lem:2} Recall that $W_{N}(x)=N^{\beta}W(N^{\beta/d}x)$. Then%
\[
||W_{N}||_{L^{2}}^{2}\leq CN^{\beta}%
\]%
\[
||W_{N}||_{W^{\gamma,2}}^{2}\leq CN^{\gamma^{\ast}}%
\]
with $\gamma^{\ast}=\frac{\beta}{d}(2\gamma+d)$.
\end{lem}

We shall use also the following tightness result.

\begin{lem}
\label{lem:3} Let $X_{1}$ and $X_{2}$ be two metric spaces with their Borel
$\sigma$-fields $\mathcal{B}_{1},\mathcal{B}_{1}$ and let $\phi:X_{1}%
\rightarrow X_{2}$ be a continuous function. Let $\mathcal{G}_{1}$ be a family
of probability measures on $\left(  X_{1},\mathcal{B}_{1}\right)  $. Denote by
$\mathcal{G}_{2}$ the family of probability measures on $\left(
X_{2},\mathcal{B}_{2}\right)  $ obtained as image laws of the measures in
$\mathcal{G}_{1}$ under the map $\phi$. If $\mathcal{G}_{1}$ is tight, then
$\mathcal{G}_{2}$ is tight.
\end{lem}

\begin{proof}
Given $\epsilon>0$, let $K_{1}^{\epsilon}\subset X_{1}$ be a compact set such
that $\mu\left(  K_{1}^{\epsilon}\right)  >1-\epsilon$ for every $\mu
\in\mathcal{G}_{1}$. Set $K_{2}^{\epsilon}=\phi\left(  K_{1}^{\epsilon
}\right)  $; it is a compact set of $X_{2}$ and for every $\nu\in
\mathcal{G}_{2}$, called $\mu$ a measure in $\mathcal{G}_{1}$ such that $\nu$
is the image of $\mu$ under $\phi$, we have
\[
\nu\left(  K_{2}^{\epsilon}\right)  =\mu\left(  K_{1}^{\epsilon}\right)
>1-\epsilon.
\]
This proves tightness of $\mathcal{G}_{2}$.
\end{proof}
Regarding the initial condition, we state a result, that will be usefull in the proof of tightness
\begin{lem}\label{lem:initial_condition}
Assume that $X^i_0$ , $i = 1, ...,N,$ are independent identically distributed r.v
with common probability density $u_0\in C^2(\mathbb{R}^d)$, then on $u^N_0$, defined as $u^N_0(x)=(W^N\ast u_0)(x)$, we get the following uniform bounds for $p>1$:
\[\mathbb{E}\left[ \left|\left|u^N_0\right|\right|^p_{W^{\alpha,2}}\right]\leq C_{u_0,\alpha,p}\]
where $C$ is a constant depending on $p$ and $\alpha$.
\end{lem}
\begin{proof}
By the definition of the norm in the fractional Sobolev space, we need to estimate uniformly in N:
\begin{equation}
\label{eq:u0}
\mathbb{E}\left[ \left|\left|u^N_0\right|\right|^p_{W^{\alpha,2}}\right]=\mathbb{E}\left[  \left\Vert u_{0}^{N}\right\Vert _{L^{2}\left(
\mathbb{R}^{d}\right)  }^{p}\right]  +\mathbb{E}\left[  \left\Vert \left(
-\Delta\right)  ^{\alpha}u_{0}^{N}\right\Vert _{L^{2}\left(  \mathbb{R}%
^{d}\right)  }^{p}\right]  .
\end{equation}

We recall that $u_0$ is compactly supported and moreover fractional laplacian is a local operator, namely it preserves compactness properties of functions.
Then, for $p\geq2$,%
\[
\left\Vert u_{0}^{N}\right\Vert _{L^{2}\left(  \mathbb{R}^{d}\right)  }%
^{p}\leq\int_{\mathcal{B}_1}\left\vert u_{0}^{N}\left(  x\right)  \right\vert ^{p}dx
\]
\[
\left\Vert \left(  -\Delta\right)  ^{\alpha}u_{0}^{N}\right\Vert
_{L^{2}\left(  \mathbb{R}^{d}\right)  }^{p}\leq 
\int_{\mathcal{B}_1}\left\vert \left(  -\Delta\right)
^{\alpha}u_{0}^{N}\left(  x\right)  \right\vert ^{p}dx
\]
where $\mathcal{B}_1$, $\mathcal{B}_2$ are respectively compact supports of $u^N_0$ and $\left(  -\Delta\right)
^{\alpha}u_{0}^{N}$ Assuming that
\[
Y^{i}=Y^{i}\left(  x\right)  =W_{N}\left(  x-X_{0}^{i}\right),
\]
\[
\widetilde{Y}^{i}=\widetilde{Y}^{i}\left(  x\right)  =\left(  -\Delta\right)
^{\alpha}W_{N}\left(  x-X_{0}^{i}\right),
\]
we can write the estimates for \eqref{eq:u0} in the following terms:
\[
\leq \mathbb{E}\left[  \int_{\mathcal{B}_1}\left\vert \frac{1}%
{N}\sum_{i=1}^{N}Y^{i}\left(  x\right)  \right\vert ^{p}dx\right]
+\mathbb{E}\left[  \int_{\mathcal{B}_2}\left\vert \frac{1}%
{N}\sum_{i=1}^{N}\widetilde{Y}^{i}\left(  x\right)  \right\vert ^{p}dx\right]
\]
Then we need to estimate
\begin{equation*}
\mathbb{E}\left[  \left\vert \frac{1}{N}\sum_{i=1}^{N}Y^{i}\left(  x\right)
\right\vert ^{p}\right]  +\mathbb{E}\left[  \left\vert \frac{1}{N}\sum
_{i=1}^{N}\widetilde{Y}^{i}\left(  x\right)  \right\vert ^{p}\right].\label{prove}%
\end{equation*}
Being $Y^i\geq 0$ on the first summand we have:
\begin{multline*}
\mathbb{E}\left[  \left\vert \frac{1}{N}\sum_{i=1}^{N}Y^{i}\right\vert
^{p}\right]     =\int_{0}^{\infty}P\left(  \left(  \frac{1}{N}\sum_{i=1}%
^{N}Y^{i}\right)  ^{p}>t\right)  dt\\
=\int_{0}^{\infty}P\left(  \frac{1}{N}%
\sum_{i=1}^{N}Y^{i}>t^{1/p}\right)  dt\\
  =\int_{0}^{\infty}P\left(  \exp\left(  \frac{1}{N}\sum_{i=1}^{N}%
Y^{i}\right)  >\exp\left(  t^{1/p}\right)  \right)  dt\\
 \leq\int_{0}^{\infty}\exp\left(  -t^{1/p}\right)  \mathbb{E}\left[
e^{\frac{1}{N}\sum_{i=1}^{N}Y^{i}}\right]  dt\\
  =e^{N\log\mathbb{E}\left[  e^{\frac{Y}{N}}\right]  }\int_{0}^{\infty}%
\exp\left(  -t^{1/p}\right)  dt
\end{multline*}
where $Y$ has the same law of $Y^i$. Notice that the equality 
\[\mathbb{E}\left[e^{\frac{1}{N}\sum_{i=1}^{N}Y^{i}}\right]=e^{N\log\mathbb{E}\left[  e^{\frac{Y}{N}}\right] }\]
follows easily from the fact that $Y^i$ are iid.
Because also $\tilde{Y}^i\geq 0$, the same result holds for the second term. Then
\begin{multline*}
\mathbb{E}\left[  \left\vert \frac{1}{N}\sum_{i=1}^{N}Y^{i}\left(  x\right)
\right\vert ^{p}\right]  +\mathbb{E}\left[  \left\vert \frac{1}{N}\sum
_{i=1}^{N}\widetilde{Y}^{i}\left(  x\right)  \right\vert ^{p}\right]\leq\\
\left(e^{N\log\mathbb{E}\left[  e^{\frac{Y}{N}}\right]  }+e^{N\log\mathbb{E}\left[  e^{\frac{\tilde{Y}}{N}}\right]  }\right)\int_{0}^{\infty}%
\exp\left(  -t^{1/p}\right)  dt
\end{multline*}

Let us estimate the first term (the same result will hold for the second term).
We recall some basics inequalities $\log\left(  1+x\right)
\leq x$, $e^{x}-1\leq xe^{x}$ for $x\geq0$. Then
\begin{align*}
\log\mathbb{E}\left[  e^{\frac{Y\left(  x\right)  }{N}}\right]    &
=\log\left(  1+\mathbb{E}\left[  e^{\frac{Y\left(  x\right)  }{N}}-1\right]
\right)  \\
& \leq\mathbb{E}\left[  e^{\frac{Y\left(  x\right)  }{N}}-1\right]  \\
& \leq\mathbb{E}\left[  \frac{Y\left(  x\right)  }{N}e^{\frac{Y\left(
x\right)  }{N}}\right]
\end{align*}
We have to estimate:
\[\mathbb{E}\left[  Y\left(  x\right)e^{\frac{Y\left(
x\right)  }{N}}\right]\quad\textrm{and}\quad \mathbb{E}\left[  \tilde{Y}\left(  x\right) e^{\frac{\tilde{Y}\left(
x\right)  }{N}}\right]\]
Recalling the definition of $Y^i$,
\[
\frac{Y\left(  x\right)  }{N}=N^{-1}W_{N}\left(  x-X_{0}\right)  =N^{\beta
-1}W\left(  N^{\beta/d}\left(  x-X_{0}\right)  \right)  \leq C.
\]
being $W$ bounded, we get that $\frac{Y(x)}{N}$ is bounded. Now we just need to estimate $\mathbb{E}\left[Y(x)\right]$.
\[\mathbb{E}\left[Y(x)\right]=\left(W^N\ast u_0\right)(x)\]
the last term is bounded because $u_0$ is it. Let us analyze the second term, which is a bit more delicate. By the definition of $\tilde{Y}^i$,
\[
\frac{\widetilde{Y}\left(  x\right)  }{N}=N^{-1}\left(  -\Delta\right)
^{\alpha}W_{N}\left(  x-X_{0}\right)  \leq CN^{-1+\beta}N^{2\alpha\beta/d}.
\]
Choosing an $\alpha$ small enough the term $\frac{\widetilde{Y}\left(  x\right)  }{N}$ is bounded. 
At the end we need to prove a uniform estimate on $\mathbb{E}\left[\tilde{Y}(x)\right]$.
\begin{align*}
\mathbb{E}\left[  \widetilde{Y}\left(  x\right)  \right]    & =\mathbb{E}%
\left[  \left(  -\Delta\right)  ^{\alpha}W_{N}\left(  x-X_{0}\right)  \right]
=\int\left[  \left(  -\Delta\right)  ^{\alpha}W_{N}\right]  \left(
x-x_{0}\right)  u_{0}\left(  x_{0}\right)  dx_{0}\\
& \overset{x_{0}^{\prime}=x-x_{0}}{=}-\int\left(  -\Delta\right)  ^{\alpha
}W_{N}\left(  x_{0}^{\prime}\right)  u_{0}\left(  x-x_{0}^{\prime}\right)
dx_{0}^{\prime}\\
& =-\left\langle \left(  -\Delta\right)  ^{\alpha}W_{N},u_{0}\left(
x-\cdot\right)  \right\rangle _{L^{2}}\\
& =-\left\langle W_{N},\left(  -\Delta\right)  ^{\alpha}u_{0}\left(
x-\cdot\right)  \right\rangle _{L^{2}}\\
& =-\int W_{N}\left(  x_{0}^{\prime}\right)  \left[  \left(  -\Delta\right)
^{\alpha}u_{0}\right]  \left(  x-x_{0}^{\prime}\right)  dx_{0}^{\prime}\\
& =\left[  W_{N}\ast\left(  \left(  -\Delta\right)  ^{\alpha}u_{0}\right)
\right]  \left(  x\right)  .
\end{align*}
Being $\left(  -\Delta\right)  ^{\alpha}u_{0}$ compactly supported and continuous also $W_{N}\ast\left(\left(  -\Delta\right)  ^{\alpha}u_{0}\right)  $ is uniformly bounded. In summary, 
\[\mathbb{E}\left[ \left|\left|u^N_0\right|\right|^p_{W^{\alpha,2}}\right]\leq C_{\mathcal{B}_1,\mathcal{B}_2,u_0,\alpha,p}.\]
\end{proof}

\section{Tightness}

\subsection{Compactness of function spaces\label{sect comp}}

We use Corollary 9 of J. Simon \cite{Simon}, using as far as possible the
notations of that paper, for easiness of reference. Given a ball
$B_{R}:=B\left(  0,R\right)  $ in $\mathbb{R}^{d}$, taken $\alpha>\epsilon>0$,
consider the spaces%
\[
X=W^{\alpha,2}\left(  B_{R}\right)  ,\qquad B=W^{\alpha-\epsilon,2}\left(
B_{R}\right)  ,\qquad Y=W^{-2,2}\left(  B_{R}\right)  .
\]
We have%
\[
X\subset B\subset Y
\]
with compact dense embeddings. Moreover, we have the interpolation inequality (see Theorem 6.4.5 in \cite{BerLof}) 
\[
\left\Vert f\right\Vert _{B}\leq C_{R}\left\Vert f\right\Vert _{X}^{1-\theta
}\left\Vert f\right\Vert _{Y}^{\theta}%
\]
for all $f\in X$, with
\[
\theta=\frac{\epsilon}{2+\alpha}.
\]
These are preliminary assumptions of Corollary 9 of \cite{Simon}. The
Corollary tells us that the embedding of
\[
\mathcal{W}_{R}:=L^{r_{0}}\left(  0,T;X\right)  \cap W^{s_{1},r_{1}}\left(
0,T;Y\right)
\]
is relatively compact in $C\left(  \left[  0,T\right]  ;B\right)  $, if
$s_{1}r_{1}>1$ and $r_{0}$ is so large that $s_{\theta}>\frac{1}{r_{\theta}}$
where (always following the notations of \cite{Simon}) $s_{\theta}=\theta
s_{1}$, $\frac{1}{r_{\theta}}=\frac{1-\theta}{r_{0}}+\frac{\theta}{r_{1}}$.
Below we shall choose for instance $s_{1}=\frac{1}{3}$ (any number smaller
than $\frac{1}{2}$) and $r_{1}=4$, so $s_{1}r_{1}>1$ is fulfilled. Then we
need%
\[
\frac{\theta}{3}>\frac{1-\theta}{r_{0}}+\frac{\theta}{4}.
\]
The logical sequence of our choices is: given $\beta\in\left(  0,1\right)  $
(think to $\beta$ close to 1, which is the most difficult choice), we shall
choose $\alpha>0$ so small to satisfy a condition related to $\beta$ which
appears in the proof of Lemma \ref{lem:4}\ below (when $\beta$ is close to 1,
we have to choose $\alpha$ small). Given this small $\alpha$, we choose
$\epsilon\in\left(  0,\alpha\right)  $ and then $\theta=\frac{\epsilon
}{2+\alpha}$ is determined, typically very small. Now, we choose $r_{0}$ so
large that $\frac{\theta}{3}>\frac{1-\theta}{r_{0}}+\frac{\theta}{4}$.
Summarising we choose $(\alpha,s_1,r_1,r_0,\epsilon)$, in the following way:
\begin{equation*}
   \begin{cases}
\alpha: \textrm{ determined by } \beta\\
(s_1,r_1): \textrm{determined (almost) a priori. See in Proposition \ref{bound_tightness} condition } \\
\quad \quad\quad\quad\quad\quad\quad\quad\quad\quad\quad\quad\quad\quad\quad\quad\quad\quad\quad\quad\quad\quad\quad\quad\quad\quad\quad\quad s_1r_1-\frac{r_1}{2}<0\\
r_0: \textrm{large enough s.t. }\theta s_1> \frac{1-\theta}{r_0}+\frac{\theta}{r_1}\\
\epsilon: \epsilon<\alpha \textrm{ arbitrarily small}
   \end{cases}
\end{equation*}
The final step consists in taking $\mathbb{R}^{d}$ instead of $B_{R}$. We
denote by $W_{loc}^{\alpha,2}\left(  \mathbb{R}^{d}\right)  $ the space of
functions $f\in\cap_{R>0}W^{\alpha,2}\left(  B_{R}\right)  $ and we endow this
space with the metric%
\[
d_{W_{loc}^{\alpha,2}}\left(  f,g\right)  =\sum_{n=1}^{\infty}2^{-n}\left(
\left\Vert f-g\right\Vert _{W^{\alpha,2}\left(  B_{n}\right)  }\wedge1\right)
.
\]
Under the same conditions on the indexes, we have now that
\[
\mathcal{W}:=L^{r_{0}}\left(  0,T;W^{\alpha,2}\left(  \mathbb{R}^{d}\right)
\right)  \cap W^{s_{1},r_{1}}\left(  0,T;W^{-2,2}\left(  \mathbb{R}%
^{d}\right)  \right)
\]
is compactly embedded into $C\left(  \left[  0,T\right]  ;W_{loc}%
^{\alpha-\epsilon,2}\left(  \mathbb{R}^{d}\right)  \right)  $.

\subsection{Main estimate on the empirical density $u^{N}$}
Before looking into details for the derivation of main estimates for the empirical density, we state the mild formulation for $u_{t}^{N}$, see Lemma \ref{lem:1} for the identity for $u_{t}^{N}$:
\[
u_{t}^{N}=e^{tA}u_{0}^{N}+\int_{0}^{t}e^{(t-s)A}\mathrm{div}(W_{N}\ast\left(  b(u_{s}^{N},m_{s}^{N}%
)S_{s}^{N}\right)  )(x)ds+\int%
_{0}^{t}e^{(t-s)A}dM_{s}^{N}%
\]
\begin{lem}
\label{lem:4}Given $\beta\in\left(  0,1\right)  $, there exists $\alpha>0$
small enough such that the following holds:\ for every $p>1$ there is a
constant $C_{p}>0$ such that%
\[
\sup_{t\in\left[  0,T\right]  }\mathbb{E}\left[  \left\Vert u_{t}%
^{N}\right\Vert _{W^{\alpha,2}(\mathbb{R}^{d})}^{p}\right]  \leq C_{p}%
\]
independently of $N$.
\end{lem}

\begin{proof}
\textbf{Step 1} (preliminary estimates). We shall use the equivalence between
norms (\ref{prop1_semigroup}):
\[
||(I-A)^{\alpha/2}f||_{L^{2}(\mathbb{R}^{d})}\sim||f||_{W^{\alpha
,2}(\mathbb{R}^{d})}.
\]
Then, up to a constant, denoting with $f_{s}^{N}(x)=\mathrm{div}(W_{N}\ast\left(  b(u_{s}^{N},m_{s}^{N}%
)S_{s}^{N}\right)  )(x)$
\begin{align*}
||u_{t}^{N}||_{W^{\alpha,2}(\mathbb{R}^{d})}  &  \leq||(I-A)^{\alpha
/2}e^{tA}u_{0}^{N}||_{L^{2}(\mathbb{R}^{d})}+||(I-A)^{\alpha/2}%
\int_{0}^{t}e^{(t-s)A}f_{s}^{N}ds||_{L^{2}(\mathbb{R}^{d})}\\
&  +||(I-A)^{\alpha/2}\int_{0}^{t}e^{(t-s)A}dM_{s}^{N}%
||_{L^{2}(\mathbb{R}^{d})}.
\end{align*}
On the first term, using (\ref{prop2_semigroup}) we prove the following estimate
\begin{align*}
&  \mathbb{E}\left[  ||(I-A)^{\alpha/2}e^{tA}u_{0}^{N}||_{L^{2}%
(\mathbb{R}^{d})}^{p}\right] \\
&  \leq||e^{tA}||_{L^{2}\rightarrow L^{2}}^{p}\mathbb{E}\left[
||(I-A)^{\alpha/2}u_{0}^{N}||_{L^{2}(\mathbb{R}^{d})}^{p}\right]  \leq C\mathbb{E}\left[
\left\Vert u_{0}^{N}\right\Vert _{W^{\alpha,2}(\mathbb{R}^{d})}^{p}\right]  .
\end{align*}
The last expected value is bounded by the assumption that $u_{0}$ is $C^{1}$
compact support: in this case one can show convergence of the empirical means
of the i.i.d. r.v.'s $X_{0}^{i}$ , that imply a uniform in $N$ bound on
$\mathbb{E}\left[  \left\Vert u_{0}^{N}\right\Vert _{W^{\alpha,2}%
(\mathbb{R}^{d})}^{p}\right]  $, for every $p$ (see
\cite{FlandoliLeimbachOlivera} for similar results).

On the third term, we use the following fact. For every $p>1$ there is a
constant $C_{p}>0$ such that, if $\Phi_{t}^{1},...,\Phi_{t}^{N}$ are adapted
square integrable processes with values in a Hilbert space $H$,
\[
\mathbb{E}\left[  \left\Vert \sum_{i=1}^{N}\int_{0}^{T}\Phi_{t}^{i}dB_{t}%
^{i}\right\Vert _{H}^{p}\right]  \leq C_{p}\mathbb{E}\left[  \left(
\sum_{i=1}^{N}\int_{0}^{T}\left\Vert \Phi_{t}^{i}\right\Vert _{H}%
^{2}dt\right)  ^{p/2}\right]  .
\]
Therefore%
\begin{align*}
&  \mathbb{E}\left[  \left\Vert \int_{0}^{t}(I-A)^{\alpha/2}%
e^{(t-s)A}dM_{s}^{N}\right\Vert _{L^{2}(\mathbb{R}^{d})}^{p}\right] \\
&  =\mathbb{E}\left[  \left\Vert \frac{\sigma}{N}\sum_{i=1}^{N}\int_{0}%
^{t}(I-A)^{\alpha/2}e^{(t-s)A}\nabla W_{N}(\cdot-X_{s}^{i,N}%
)dB_{s}^{i}\right\Vert _{L^{2}(\mathbb{R}^{d})}^{p}\right] \\
&  \leq C_{p}\mathbb{E}\left[  \left(  \frac{\sigma^2}{N^{2}}\sum_{i=1}^{N}\int%
_{0}^{T}\left\Vert (I-A)^{\alpha/2}e^{(t-s)A}\nabla W_{N}%
(\cdot-X_{s}^{i,N})\right\Vert _{L^{2}(\mathbb{R}^{d})}^{2}ds\right)
^{p/2}\right] \\
&  =C_{p}\mathbb{E}\left[  \left(  \frac{\sigma^2}{N^{2}}\sum_{i=1}^{N}\int_{0}%
^{T}\left\Vert (I-A)^{\alpha/2}e^{(t-s)A}\nabla W_{N}\right\Vert
_{L^{2}(\mathbb{R}^{d})}^{2}ds\right)  ^{p/2}\right] \\
&  =C_{p}\left(  \frac{\sigma^2}{N}\int_{0}^{T}\left\Vert (I-A)^{\alpha
/2}e^{(t-s)A}\nabla W_{N}\right\Vert _{L^{2}(\mathbb{R}^{d})}%
^{2}ds\right)  ^{p/2}.
\end{align*}
Moreover, the gradient commutes with the heat semigroup and the fractional
powers of the Laplacian. Hence, using (\ref{prop2_semigroup}) and (\ref{prop3_semigroup}), the
integrand can be estimated as follows%
\begin{align*}
&  \left\Vert (I-A)^{\alpha/2}e^{(t-s)A}\nabla W_{N}\right\Vert
_{L^{2}(\mathbb{R}^{d})}^{2}\\
&  \leq\left[  ||\nabla(I-A)^{-1/2}||_{L^{2}\rightarrow L^{2}}%
||(I-A)^{\frac{1-\epsilon}{2}}e^{(t-s)A}||_{L^{2}\rightarrow L^{2}%
}||(I-A)^{\frac{\alpha+\epsilon}{2}}W_{N}||_{L^{2}(\mathbb{R}^{d}%
)}\right]  ^{2}\\
&  \leq\frac{c}{(t-s)^{1-\epsilon}}||W_{N}||_{W^{\alpha+\epsilon,2}%
(\mathbb{R}^{d})}^{2}.
\end{align*}
From Lemma \ref{lem:2} we get%
\[
\left\Vert (I-A)^{\alpha/2}e^{(t-s)A}\nabla W_{N}\right\Vert
_{L^{2}(\mathbb{R}^{d})}^{2}\leq\frac{c}{(t-s)^{1-\epsilon}}N^{(\alpha
+\epsilon)^{\ast}}%
\]

and thus we can estimate the martingale term in the following way:%
\begin{align*}
&  \mathbb{E}\left[  \left\Vert \int_{0}^{t}(I-A)^{\alpha/2}%
e^{(t-s)A}dM_{s}^{N}\right\Vert _{L^{2}(\mathbb{R}^{d})}^{p}\right] \\
&  \leq C_{p}\left(  \frac{\sigma^2}{N}\int_{0}^{T}\frac{c}{(t-s)^{1-\epsilon}%
}N^{(\alpha+\epsilon)^{\ast}}ds\right)  ^{p/2}\\
&  =C_{p,T}\left(  \frac{N^{(\alpha+\epsilon)^{\ast}}}{N}\right)  ^{p/2}.
\end{align*}

Choosing $\alpha$ so small that $(\alpha+\epsilon)^{\ast}\leq1$, i.e.
$\beta\leq\frac{d}{2(\alpha+\epsilon)+d}<1$, we get a uniform bound on the
martingale term.

Finally, thanks to the boundness on $b$ we get the following estimate%
\begin{align*}
&  \left\vert W_{N}\ast\left(  b(u_{t}^{N},m_{t}^{N})S_{t}^{N}\right)  \left(
x\right)  \right\vert \\
&  =\left\vert \int_{\mathbb{R}^{d}}W_{N}\left(  x-y\right)  b(u_{t}^{N}%
,m_{t}^{N})\left(  y\right)  S_{t}^{N}\left(  dy\right)  \right\vert \\
&  \leq\int_{\mathbb{R}^{d}}W_{N}\left(  x-y\right)  \left\vert b(u_{t}%
^{N},m_{t}^{N})\left(  y\right)  \right\vert S_{t}^{N}\left(  dy\right) \\
&  \leq C\int_{\mathbb{R}^{d}}W_{N}\left(  x-y\right)  S_{t}^{N}\left(
dy\right) \\
&  =Cu_{t}^{N}\left(  x\right)
\end{align*}
hence
\[
||W_{N}\ast\left(  b(u_{t}^{N},m_{t}^{N})S_{t}^{N}\right)  ||_{L^{2}%
(\mathbb{R}^{d})}^{2}\leq C||u_{t}^{N}||_{L^{2}(\mathbb{R}^{d})}%
\]
\newline\textbf{Step 2} (estimate in $L^{2}(\mathbb{R}^{d})$). Consider the
case $\alpha=0$ in the previous computations. We have proved, with the
notation $H=L^{2}(\mathbb{R}^{d})$, that
\[
\left\Vert u_{t}^{N}\right\Vert _{L^{p}\left(  \Omega;H\right)  }\leq
C+\left\Vert \int_{0}^{t}e^{(t-s)A}f_{s}^{N}ds\right\Vert _{L^{p}\left(
\Omega;H\right)  }.
\]
Thus%
\[
\left\Vert u_{t}^{N}\right\Vert _{L^{p}\left(  \Omega;H\right)  }\leq
C+\int_{0}^{t}\left\Vert e^{(t-s)A}f_{s}^{N}\right\Vert _{L^{p}\left(
\Omega;H\right)  }ds.
\]
We have%
\begin{align*}
\left\Vert e^{(t-s)A}f_{s}^{N}\right\Vert _{L^{p}\left(  \Omega;H\right)
}  &  =\mathbb{E}\left[  \left\Vert \nabla\cdot e^{(t-s)A}(W_{N}%
\ast\left(  b(u_{t}^{N},m_{t}^{N})S_{t}^{N}\right)  \right\Vert _{L^{2}%
(\mathbb{R}^{d})}^{p}\right]  ^{1/p}\\
&  \leq\left\Vert \nabla\cdot e^{(t-s)A}\right\Vert _{L^{2}\rightarrow
L^{2}}\mathbb{E}\left[  \left\Vert W_{N}\ast b(u_{t}^{N},m_{t}^{N})S_{t}%
^{N}\right\Vert _{L^{2}(\mathbb{R}^{d})}^{p}\right]  ^{1/p}\\
&  \leq\frac{C}{(t-s)^{\frac{1}{2}}}\left\Vert u_{s}^{N}\right\Vert
_{L^{p}\left(  \Omega;H\right)  }%
\end{align*}
using properties on the analytical semigroup \eqref{prop2_semigroup}, \eqref{prop3_semigroup} and the last bound of Step 1. Therefore%
\[
\left\Vert u_{t}^{N}\right\Vert _{L^{p}\left(  \Omega;H\right)  }\leq
C+\int_{0}^{t}\frac{C}{(t-s)^{\frac{1}{2}}}\left\Vert u_{s}^{N}\right\Vert
_{L^{p}\left(  \Omega;H\right)  }ds.
\]
A generalised form of Gronwall lemma implies%
\[
\sup_{t\in\left[  0,T\right]  }\left\Vert u_{t}^{N}\right\Vert _{L^{p}\left(
\Omega;H\right)  }\leq C
\]
where the constant $C$ depends on $p$ but not on $N$.

\textbf{Step 3} (estimate in $W^{\alpha,2}(\mathbb{R}^{d})$). Similarly to the
beginning of Step 2, we have%
\[
\left\Vert u_{t}^{N}\right\Vert _{L^{p}\left(  \Omega;\widetilde{H}\right)
}\leq C+\int_{0}^{t}\left\Vert e^{(t-s)A}f_{s}^{N}\right\Vert
_{L^{p}\left(  \Omega;\widetilde{H}\right)  }ds
\]
where now $\widetilde{H}=W^{\alpha,2}(\mathbb{R}^{d})$;\ and recalling some properties of the analytical semigroup, see \eqref{prop1_semigroup},\eqref{prop2_semigroup}, \eqref{prop3_semigroup} similarly we get, 
\[
\left\Vert e^{(t-s)A}f_{s}^{N}\right\Vert _{L^{p}\left(  \Omega
;\widetilde{H}\right)  }\leq\frac{C}{(t-s)^{\frac{\alpha+1}{2}}}\left\Vert
u_{s}^{N}\right\Vert _{L^{p}\left(  \Omega;H\right)  }.
\]
But from Step 2 we know that $\left\Vert u_{s}^{N}\right\Vert _{L^{p}\left(
\Omega;H\right)  }$ is uniformly bounded, hence for $\alpha<1$ we deduce the
claim of the Lemma.
\end{proof}

\subsection{Tightness of $\left(  u_{t}^{N},m_{t}^{N}\right)  $}

Recall from Section \ref{sect comp} that the space there denoted by
$\mathcal{W}$ is compactly embedded into $C\left(  \left[  0,T\right]
;W_{loc}^{\alpha-\epsilon,2}\left(  \mathbb{R}^{d}\right)  \right)  $, when
$s_{1}=\frac{1}{3}$ and $r_{1}=4$ and when, having chosen $\alpha$ small
enough related to the original choice of $\beta$ in order that the result of
Lemma \ref{lem:4} is true, we take $r_{0}$ large enough.

In order to prove tightness of the family of laws of $u_{t}^{N}$ in
$\mathcal{W}$ we have to prove that $u_{t}^{N}$ is bounded in probability in
$L^{r_{0}}\left(  0,T;W^{\alpha,2}\left(  \mathbb{R}^{d}\right)  \right)  $
and in $W^{s_{1},r_{1}}\left(  0,T;W^{-2,2}\left(  \mathbb{R}^{d}\right)
\right)  $. For the first claim it is sufficient to prove that%
\[
\mathbb{E}\int_{0}^{T}\left\Vert u_{t}^{N}\right\Vert _{W^{\alpha,2}\left(
\mathbb{R}^{d}\right)  }^{r_{0}}dt\leq C
\]
and this is true by Lemma \ref{lem:4}, because%
\[
\mathbb{E}\int_{0}^{T}\left\Vert u_{t}^{N}\right\Vert _{W^{\alpha,2}\left(
\mathbb{R}^{d}\right)  }^{r_{0}}dt=\int_{0}^{T}\mathbb{E}\left[  \left\Vert
u_{t}^{N}\right\Vert _{W^{\alpha,2}\left(  \mathbb{R}^{d}\right)  }^{r_{0}%
}\right]  dt\leq\sup_{t\in\left[  0,T\right]  }\mathbb{E}\left[  \left\Vert
u_{t}^{N}\right\Vert _{W^{\alpha,2}(\mathbb{R}^{d})}^{r_{0}}\right]  .
\]
The second claim is proved in the next Proposition.

\begin{prop}\label{bound_tightness}
\label{prop:1}The family $\{u_{t}^{N}\}_{N}$ is bounded in probability in
$W^{s_{1},r_{1}}\left(  0,T;W^{-2,2}\left(  \mathbb{R}^{d}\right)  \right)  $.
\end{prop}

\begin{proof}
Let us recall that a norm on $W^{s_{1},r_{1}}\left(  0,T;W^{-2,2}\left(
\mathbb{R}^{d}\right)  \right)  $ is given by the sum%
\[
\left(  \int_{0}^{T}\left\Vert f_{t}\right\Vert _{W^{-2,2}\left(
\mathbb{R}^{d}\right)  }^{r_{1}}dt\right)  ^{1/r_{1}}+\left(  \int_{0}^{T}%
\int_{0}^{T}\frac{\left\Vert f_{t}-f_{s}\right\Vert _{W^{-2,2}\left(
\mathbb{R}^{d}\right)  }^{r_{1}}}{\left\vert t-s\right\vert ^{1+s_{1}r_{1}}%
}dtds\right)  ^{1/r_{1}}.
\]
The property
\[
\mathbb{E}\int_{0}^{T}\left\Vert u_{t}^{N}\right\Vert _{W^{-2,2}\left(
\mathbb{R}^{d}\right)  }^{r_{1}}dt\leq C
\]
is a consequence of Lemma \ref{lem:4}, because $\left\Vert u_{t}%
^{N}\right\Vert _{W^{-2,2}}$ is a weaker norm than $\left\Vert u_{t}%
^{N}\right\Vert _{W^{\alpha,2}}$. We have to prove%
\[
\mathbb{E}\int_{0}^{T}\int_{0}^{T}\frac{\left\Vert u_{t}^{N}-u_{s}%
^{N}\right\Vert _{W^{-2,2}\left(  \mathbb{R}^{d}\right)  }^{r_{1}}}{\left\vert
t-s\right\vert ^{1+s_{1}r_{1}}}dtds\leq C.
\]

Thus, for $t>s$, we have to estimate
\[
\mathbb{E}\left[  ||u_{t}^{N}-u_{s}^{N}||_{W^{-2,2}(\mathbb{R}^{d})}^{r_{1}%
}\right]  .
\]
From the equation satisfied by $u_{t}^{N}$, proved in lemma \ref{lem:1} and H{\"{o}}lder inequality, we
have%
\begin{align*}
\mathbb{E}\left[  ||u_{t}^{N}-u_{s}^{N}||_{W^{-2,2}(\mathbb{R}^{d})}^{r_{1}%
}\right]   &  \leq C(t-s)^{r_{1}-1}\mathbb{E}\int_{s}^{t}||A u_{r}%
^{N}||_{W^{-2,2}(\mathbb{R}^{d})}^{r_{1}}dr+C\mathbb{E}\left[  ||M_{t}%
^{N}-M_{s}^{N}||_{W^{-2,2}(\mathbb{R}^{d})}^{r_{1}}\right] \\
&  +C(t-s)^{r_{1}-1}\mathbb{E}\int_{s}^{t}||\text{div}(W_{N}\ast b(u_{r}%
^{N},m_{r}^{N})S_{r}^{N})||_{W^{-2,2}(\mathbb{R}^{d})}^{r_{1}}dr
\end{align*}
Being $A$ a bounded operator from $L^{2}$ to $W^{-2,2}$, we have%
\[
C(t-s)^{r_{1}-1}\mathbb{E}\int_{s}^{t}||A u_{r}^{N}||_{W^{-2,2}%
(\mathbb{R}^{d})}^{r_{1}}dr\leq C(t-s)^{r_{1}-1}\int_{s}^{t}\mathbb{E}\left[
||u_{r}^{N}||_{L^{2}(\mathbb{R}^{d})}^{r_{1}}\right]  dr\leq C(t-s)^{r_{1}}%
\]
thanks to the estimate of Lemma \ref{lem:4}. Notice that the spaces $L^{2}$
and $W^{1,2}$ are continuously embedded in $W^{-2,2}$, namely there exists a
constant $C>0$ such that $||f||_{W^{-2,2}}\leq C||f||_{L^{2}}$ and
$||f||_{W^{-2,2}}$ $\leq C||f||_{W^{-1,2}}$. We shall use this in the next
computations. We have%
\begin{align*}
&  ||\text{div}(W_{N}\ast b(u_{r}^{N},m_{r}^{N})S_{r}^{N})||_{W^{-2,2}%
(\mathbb{R}^{d})}\\
&  \leq C||(W_{N}\ast b(u_{r}^{N},m_{r}^{N})S_{r}^{N}||_{W^{-1,2}%
(\mathbb{R}^{d})}\\
&  \leq C||(W_{N}\ast b(u_{r}^{N},m_{r}^{N})S_{r}^{N})||_{L^{2}(\mathbb{R}%
^{d})}\\
&  \leq C\left\Vert u_{r}^{N}\right\Vert _{L^{2}(\mathbb{R}^{d})}%
\end{align*}
where the last inequality is similar to one proved in Lemma \ref{lem:4}. Hence%
\begin{align*}
&  C(t-s)^{r_{1}-1}\mathbb{E}\int_{s}^{t}||\text{div}(W_{N}\ast b(u_{r}%
^{N},m_{r}^{N})S_{r}^{N})||_{W^{-2,2}(\mathbb{R}^{d})}^{r_{1}}dr\\
&  \leq C(t-s)^{r_{1}-1}\int_{s}^{t}\mathbb{E}\left[  ||u_{r}^{N}%
||_{L^{2}(\mathbb{R}^{d})}^{r_{1}}\right]  dr\leq C(t-s)^{r_{1}}%
\end{align*}
as above. Therefore, until now, we have proved
\[
\mathbb{E}\left[  ||u_{t}^{N}-u_{s}^{N}||_{W^{-2,2}(\mathbb{R}^{d})}^{r_{1}%
}\right]  \leq C(t-s)^{r_{1}}+C\mathbb{E}\left[  ||M_{t}^{N}-M_{s}%
^{N}||_{W^{-2,2}(\mathbb{R}^{d})}^{r_{1}}\right]  .
\]
Estimating the martingale as in Lemma \ref{lem:4}, we have
\begin{align*}
&  \mathbb{E}\left[  ||M_{t}^{N}-M_{s}^{N}||_{W^{-2,2}(\mathbb{R}^{d})}%
^{r_{1}}\right] \\
&  =\mathbb{E}\left[  \left\Vert \frac{\sigma}{N}\sum_{i=1}^{N}\int_{s}%
^{t}\nabla W_{N}(\cdot-X_{u}^{i,N})dB_{u}^{i}\right\Vert _{W^{-2,2}%
(\mathbb{R}^{d})}^{r_{1}}\right] \\
&  \leq C_{r_{1}}\mathbb{E}\left[  \left(  \frac{\sigma^2}{N^{2}}\sum_{i=1}^{N}%
\int_{s}^{t}\left\Vert \nabla W_{N}(\cdot-X_{u}^{i,N})\right\Vert
_{W^{-2,2}(\mathbb{R}^{d})}^{2}du\right)  ^{r_{1}/2}\right] \\
&  \leq C_{r_{1}}^{\prime}\mathbb{E}\left[  \left(  \frac{\sigma^2}{N^{2}}\sum
_{i=1}^{N}\int_{s}^{t}\left\Vert W_{N}(\cdot-X_{u}^{i,N})\right\Vert
_{L^{2}(\mathbb{R}^{d})}^{2}du\right)  ^{r_{1}/2}\right] \\
&  =C_{r_{1}}^{\prime}\mathbb{E}\left[  \left(  \frac{\sigma^2}{N}\left\Vert
W_{N}\right\Vert _{L^{2}(\mathbb{R}^{d})}^{2}\right)  ^{r_{1}/2}\right] \\
&  \leq C_{r_{1}}^{\prime}\mathbb{E}\left[  \left(  \frac{\sigma^2}{N}N^{\beta
}\right)  ^{r_{1}/2}\right]  \left(  t-s\right)  ^{r_{1}/2}%
\end{align*}
by Lemma \ref{lem:2}, hence (being $\beta<1$)%
\[
\mathbb{E}\left[  ||M_{t}^{N}-M_{s}^{N}||_{W^{-2,2}(\mathbb{R}^{d})}^{r_{1}%
}\right]  \leq C\left(  t-s\right)  ^{r_{1}/2}.
\]
Summarising,%
\[
\mathbb{E}\left[  ||u_{t}^{N}-u_{s}^{N}||_{W^{-2,2}(\mathbb{R}^{d})}^{r_{1}%
}\right]  \leq C\left(  t-s\right)  ^{r_{1}/2}.
\]
It follows that%
\begin{align*}
&  \mathbb{E}\int_{0}^{T}\int_{0}^{T}\frac{\left\Vert u_{t}^{N}-u_{s}%
^{N}\right\Vert _{W^{-2,2}\left(  \mathbb{R}^{d}\right)  }^{r_{1}}}{\left\vert
t-s\right\vert ^{1+s_{1}r_{1}}}dtds\\
&  \leq\mathbb{E}\int_{0}^{T}\int_{0}^{T}\frac{C}{\left\vert t-s\right\vert
^{1+s_{1}r_{1}-\frac{r_{1}}{2}}}dtds
\end{align*}
which is finite if $s_{1}r_{1}-\frac{r_{1}}{2}<0$; with our choice
$s_{1}=\frac{1}{3}$ and $r_{1}=4$, this is true.
\end{proof}

\begin{cor}
\label{corollary tight}The family $\{u_{t}^{N}\}_{N}$ is bounded in
probability in $\mathcal{W}$, and therefore the family of laws of $\{u_{t}%
^{N}\}_{N}$ is tight in $C\left(  \left[  0,T\right]  ;W_{loc}^{\alpha
-\epsilon,2}\left(  \mathbb{R}^{d}\right)  \right)  $. In particular, it is
tight in $C\left(  L_{loc}^{2}\right)  $. If $Q_{u}$ is any limit measure of
this family and $u$ is a r.v. with law $Q_{u}$, we also have the property%
\[
\mathbb{E}\int_{0}^{T}\left\Vert u_{t}\right\Vert _{L^{2}\left(
\mathbb{R}^{d}\right)  }^{r_{0}}dt<\infty
\]
namely $Q_{u}$ is supported on $L^{p}\left(  0,T;L^{2}(\mathbb{R}^{d})\right)
$ for some $p>2$. Therefore, if we prove that $Q_{u}$ is supported on mild
solutions, by Lemma \ref{lemma regularity} we deduce that $Q_{u}$ is supported
on $C\left(  L^{2}\right)  $.
\end{cor}

\begin{prop}
\label{prop:2} The family of laws of $\{m_{t}^{N}\}_{N}$ is tight in $C\left(
L_{loc}^{2}\right)  $. If $Q_{m}$ is any limit measure of
this family and $m$ is a r.v. with law $Q_{m}$, we also have the property%
\[
\mathbb{E}\int_{0}^{T}\left\Vert m_{t}\right\Vert _{L^{2}\left(
\mathbb{R}^{d}\right)  }^{r_{0}}dt<\infty
\]
namely $Q_{m}$ is supported on $L^{p}\left(  0,T;L^{2}(\mathbb{R}^{d})\right)
$ for some $p>2$. Therefore, if we prove that $Q_{m}$ is supported on mild
solutions, by Lemma \ref{lemma regularity} we deduce that $Q_{m}$ is supported
on $C\left(  L^{2}\right)  $.
\end{prop}

\begin{proof}
Call $C_{+}\left(  L_{loc}^{2}\right)  $ the space of nonnegative functions of
class $C\left(  L_{loc}^{2}\right)  $. Recall the explicit form of the
solution of equation (\ref{micro_ecm}) given in Proposition
\ref{prop particles}. We want to apply Lemma \ref{lem:3} with $X_{1}%
=C_{+}\left(  L_{loc}^{2}\right)  $, $X_{2}=C_{+}\left(  L_{loc}^{2}\right)
$, $\mathcal{G}_{1}$ given by the family of laws of $\{u_{t}^{N}\}_{N}$,
$\mathcal{G}_{2}$ given by the family of laws of $\{m_{t}^{N}\}_{N}$, and
$\phi$ given by (for $f\in C_{+}\left(  L_{loc}^{2}\right)  $, it is here that
we use non negativity)%
\[
\left(  \phi f\right)  _{t}\left(  x\right)  :=F_{\zeta}\left(  m_{0}%
(x),\int_{0}^{t}f_{s}(x)ds\right)
\]
where $F_{\zeta}\left(  a,b\right)  $ has been introduced in Proposition
\ref{prop particles}. Tightness of the family $\mathcal{G}_{1}$ is given by
Proposition \ref{prop:1}. To prove continuity of $\phi$, we just notice that
the map
\[
f\mapsto\int_{0}^{t}f(s,x)ds
\]
is continuous from $C\left(  L_{loc}^{2}\right)  $ to $C\left(  L_{loc}%
^{2}\right)  $ and then we have to compose with a bounded continuous map.
\end{proof}

\section{Passage to the limit}

Denote by $Q^{N}$ the law of $(u^{N},m^{N})$, on the space $C\left(
L_{loc}^{2}\right)  \times C\left(  L_{loc}^{2}\right)  $. We have proved
above that the family $\left\{  Q^{N}\right\}  $ is tight. Hence, by Prohorov
theorem, there is a subsequence $Q^{N_{k}}$ which converges weakly to some
probability measure $Q$ on $C\left(  L_{loc}^{2}\right)  \times C\left(
L_{loc}^{2}\right)  $. Moreover, from Corollary \ref{corollary tight}, the
marginal $Q_{u}$ on the first component is supported on the space
$L^{p}\left(  0,T;L^{2}(\mathbb{R}^{d})\right)  $ for some $p>2$. We want to
prove first that $Q$ is supported on the class of mild solutions of system
\eqref{eq:21}. Second, we shall prove that this class has a unique element
$\left(  u,m\right)  $; it will follow that the full sequence $\left\{
Q^{N}\right\}  $ converges to $\delta_{\left(  u,m\right)  }$ in the weak
sense of measures; and that $(u^{N},m^{N})$ converges in probability to
$\left(  u,m\right)  $, because the limit is deterministic. This will complete
the proof of Theorem \ref{main thm}; verification of the properties $u_{t}%
\geq0$ and $0\leq m_{t}\leq M$ for every $t\in\left[  0,T\right]  $ are done
with the same technique used in the proof of the next proposition, through
suitable continuous functionals; we omit the details. The regularity $C\left(
L^{2}\right)  \times C\left(  L^{2}\right)  $ of $\left(  u,m\right)  $ comes
from Corollary \ref{corollary tight} and Proposition \ref{prop:2}.
In the following proof, we will prove that $Q$ is supported on the class of mild solution of system \eqref{eq:21}. 
The proof of the following result is quite classical. It has been widely used in the mean field theory, see \cite{Sznitman}. Our case is very close to the mean field framework , but it can not be considered a particular case of the known mean field theories, in particular because of the presence of $u^N$ and the dependence of the function $g$ on a density of particles. To prove this step, we adopt the approach of \cite{KL}, see Chapter 4, although presumably it can be given along several classical lines, see \cite{Sznitman}. Before going in to some details of the proof, we introduce a family of functionals which characterizes
the solution of the system:
\begin{align*}
(u,m) &  \rightarrow\Psi_{\phi}(u,m):=\sup_{t\in\left[  0,T\right]
}\left\vert \left\langle u_{t}-u_{0},\phi\right\rangle -\frac{\sigma^2}{2}\int_{0}%
^{t}\left\langle u_{s},\Delta\phi\right\rangle ds-\int_{0}^{t}\left\langle
u_{s},b(u_{s},m_{s})\cdot\nabla\phi\right\rangle ds\right\vert \\
&  +\sup_{t\in\left[  0,T\right]  }\left\vert \langle m_{t}\left(
\cdot\right)  -F_{\zeta}\left(  m_{0}(\cdot),\int_{0}^{t}u_{s}(\cdot
)ds\right)  ,\phi\rangle\right\vert
\end{align*}
where $\phi\in C_{c}^{\infty}(\mathbb{R}^{d})$.
On these family we prove a preliminary result, to Proposition \ref{prop:passage to the limit}.

\begin{lem}\label{lem:pre_convergence}
Let $Q^{N_k}$ be the subsequence of measure of $Q^N$ that convergesp to $Q$ on $C\left(  L_{loc}^{2}\right)  \times C\left(
L_{loc}^{2}\right)  $, then
\[
\lim_{k\to \infty}Q^{N_{k}}\left(  (u,m):\Psi_{\phi}(u,m)>\delta\right)=0
\]
\end{lem}
\begin{proof}
One has that
\begin{align*}
&  Q^{N_{k}}\left(  (u,m):\Psi_{\phi}(u,m)>\delta\right)  \\
&  =\mathbb{P}\left(  \sup_{t\in\left[  0,T\right]  }\left\vert \left\langle
u_{t}^{N_{k}}-u_{0}^{N_{k}},\phi\right\rangle -\frac{\sigma^2}{2}\int_{0}^{t}\left\langle
u_{s}^{N_{k}},\Delta\phi\right\rangle ds-\int_{0}^{t}\left\langle u_{s}%
^{N_{k}},b(u_{s}^{N_{k}},m_{s}^{N_{k}})\cdot\nabla\phi\right\rangle
ds\right\vert +\right.  \\
&  +\left.  \sup_{t\in\left[  0,T\right]  }\left\vert \langle m_{t}^{N_{k}%
}\left(  \cdot\right)  -F_{\zeta}\left(  m_{0}(\cdot),\int_{0}^{t}u_{s}%
^{N_{k}}(\cdot)ds\right)  ,\phi\rangle\right\vert >\delta\right)  .
\end{align*}
The second term of the functional is clearly zero, because of the equation
satisfied by $m_{t}^{N_{k}}$. Using the identity satisfied by $u_{t}^{N_{k}}$,
we get%
\begin{align*}
& Q^{N_{k}}\left(  (u,m):\Psi_{\phi}(u,m)>\delta\right)  \\
& \leq\mathbb{P}\left(  \sup_{t\in\left[  0,T\right]  }\left\vert \left\langle
\int_{0}^{t}\left[  W_{N_{k}}\ast(b(u_{s}^{N_{k}},m_{s}^{N_{k}})S_{s}^{N_{k}%
})-u_{s}^{N_{k}}b(u_{s}^{N_{k}},m_{s}^{N_{k}})\right]  ds,\nabla
\phi\right\rangle +\left\langle M_{t}^{N_{k}},\phi\right\rangle \right\vert
>\delta\right)  .
\end{align*}

Hence it is sufficient to prove that, for given $\delta>0$, both the following
probabilities
\[
\mathbb{P}\left(  \int_{0}^{T}\left\vert \left\langle W_{N_{k}}\ast
(b(u_{s}^{N_{k}},m_{s}^{N_{k}})S_{s}^{N_{k}})-u_{s}^{N_{k}}b(u_{s}^{N_{k}%
},m_{s}^{N_{k}}),\nabla\phi\right\rangle \right\vert ds>\delta\right)
\]
and
\[
\mathbb{P}\left(  \sup_{t\in\left[  0,T\right]  }\left\vert \left\langle
M_{t}^{N_{k}},\phi\right\rangle \right\vert >\delta\right)
\]
converge to zero as $k\rightarrow\infty$. The first probability is bounded
above as follows:%
\[
\leq\mathbb{P}\left(  C\int_{0}^{T}\left\Vert W_{N_{k}}\ast(b(u_{s}^{N_{k}%
},m_{s}^{N_{k}})S_{s}^{N_{k}})-u_{s}^{N_{k}}b(u_{s}^{N_{k}},m_{s}^{N_{k}%
})\right\Vert _{L^{2}}ds>\delta\right)  .
\]
We have%
\begin{align*}
&  \left\vert \left(  W_{N_{k}}\ast(b(u_{s}^{N_{k}},m_{s}^{N_{k}})S_{s}%
^{N_{k}})\right)  \left(  x\right)  -u_{s}^{N_{k}}\left(  x\right)
b(u_{s}^{N_{k}},m_{s}^{N_{k}})\left(  x\right)  \right\vert \\
&  \leq\int W_{N_{k}}\left(  x-y\right)  \left\vert b(u_{s}^{N_{k}}%
,m_{s}^{N_{k}})\left(  x\right)  -b(u_{s}^{N_{k}},m_{s}^{N_{k}})\left(
y\right)  \right\vert S_{s}^{N_{k}}\left(  dy\right)  \\
&  \leq C^{\prime\prime}\int W_{N_{k}}\left(  x-y\right)  \left\vert
x-y\right\vert S_{s}^{N_{k}}\left(  dy\right)
\end{align*}
having used property (\ref{lip 2 b}),
\[
\leq CN_{k}^{-\beta/d}\int W_{N_{k}}\left(  x-y\right)  S_{s}^{N_{k}}\left(
dy\right)
\]
having used the form $W_{N}(x):=N^{\beta}W(N^{\beta/d}x)$ and the property of
compact support of $W$,%
\[
=CN_{k}^{-\beta/d}u_{s}^{N_{k}}\left(  x\right)  .
\]
Hence the last probability above is%
\begin{align*}
&  \leq\mathbb{P}\left(  CN_{k}^{-\beta/d}\int_{0}^{T}\left\Vert u_{s}^{N_{k}%
}\right\Vert _{L^{2}}ds>\delta\right)  \\
&  \leq\frac{CN_{k}^{-\beta/d}}{\delta}\mathbb{E}\int_{0}^{T}\left\Vert
u_{s}^{N_{k}}\right\Vert _{L^{2}}ds
\end{align*}
which goes to zero (recall the bound of Lemma \ref{lem:4}).

Finally,%
\begin{align*}
\mathbb{P}\left(  \sup_{t\in\left[  0,T\right]  }\left\vert \left\langle
M_{t}^{N_{k}},\phi\right\rangle \right\vert >\delta\right)    & \leq
\mathbb{P}\left(  \sup_{t\in\left[  0,T\right]  }\left\Vert M_{t}^{N_{k}%
}\right\Vert _{L^{2}}>\delta\right)  \\
& \leq\frac{1}{\delta^{2}}\mathbb{E}\left[  \sup_{t\in\left[  0,T\right]
}\left\Vert M_{t}^{N_{k}}\right\Vert _{L^{2}}^{2}\right]  \leq\frac{C_{T}%
}{\delta^{2}}\frac{N_k^{\epsilon\ast}}{N_k}%
\end{align*}
as in the proof of Lemma \ref{lem:4} (with $\alpha=0$, $p=2$, without
semigroup, using Doob's inequality), hence it goes to zero.
\end{proof}

\begin{prop}\label{prop:passage to the limit}
Let $Q$ be the limit probability measure of some subsequence $Q^{N_{k}}$. Then
$Q$ is supported on the set of mild solutions of system \ref{eq:21}.
\end{prop}

\begin{proof}
Firstly we observe that the functional is continuous
with respect to the topology of $C\left(  L_{loc}^{2}\right)  \times C\left(
L_{loc}^{2}\right)  $. It holds because $\phi$ is compact support with its derivatives
(this is sufficient to treat the terms $\left\langle u_{t},\phi\right\rangle $
and $\int_{0}^{t}\left\langle u_{s},\Delta\phi\right\rangle ds$), by property (\ref{lip 1 b}) (this fact plus the
previous ones is used to treat the term $\int_{0}^{t}\left\langle
u_{s},b(u_{s},m_{s})\cdot\nabla\phi\right\rangle ds$) and $F_{\zeta}$ is
continuous, $\widetilde{F}_{\zeta}$ is bounded (these facts are used to deal
with the $m$-term). Moreover $b(u,m)$ converges locally uniformly in space when $\left(
u,m\right)  $ converges in $C\left(  L_{loc}^{2}\right)  \times C\left(
L_{loc}^{2}\right) $. This last point is a delicate one, so in the next lines we will prove it. Let us consider a sequence $(u^N,m^N)$ converging in $L^2_{loc}$, fixed $\epsilon>0$ and $x\in B(0,K)$, with $K>0$
\begin{multline*}
|b(u^N,m^N)(x)-b(u,m)(x)|\leq \int_{\mathbb{R}^d} |g(|y-x|,u^N(y),m^N(y))-g(|y-x|,u(y),m(y))|dy=\\
\int_{B(0,R)} |g(|y-x|,u^N(y),m^N(y))-g(|y-x|,u(y),m^N(y))|dy+\\
\int_{B(0,R)^c} |g(|y-x|,u^N(y),m^N(y))-g(|y-x|,u(y),m(y))|dy=
I_1+I_2(R,N)
\end{multline*}
where thanks to hypothesis \eqref{hyp:g}, $R$ can be chosen such that $I_2(R,N)\leq \epsilon/2$, uniformly in $x$. Regarding $I_1$, there exists $N_0$ such that
\begin{equation*}
I_1\leq ||Dg||\int_{B(0,R)}|u^N(y)-u(y)|dy\leq  \epsilon/2\\
\end{equation*}
for all $N>N_0$. Computation including time component are straightforward. So one get that $b(u^N,m^N)$ converges locally uniformly in space to $b(u,m)$.

Thanks to continuity of the functional $\Psi_{\phi}$, by Portmanteau theorem, 
\[
Q\left(  (u,m):\Psi_{\phi}(u,m)>\delta\right)  \leq\liminf_{k}Q^{N_{k}}\left(
(u,m):\Psi_{\phi}(u,m)>\delta\right)  .
\]
Then for Lemma \ref{lem:pre_convergence}
\[
Q\left(  (u,m):\Psi_{\phi}(u,m)>\delta\right)  =0
\]
for every $\delta>0$. By a classical argument, see \cite{KL}
\[
Q\left(  (u,m):\Psi_{\phi}(u,m)=0\quad\forall\phi\in\mathcal{D}\right)  =1.
\]
Thus $Q$
is supported on weak solutions. In addition, by Corollary \ref{corollary tight}, $u$ is also of class
$L^{p}\left(  0,T;L^{2}(\mathbb{R}^{d})\right)  $ for some $p>2$. With proper
choice of $\phi$ related to the heat kernel $\frac{1}{(4\pi t)^{d/2}}%
e^{-\frac{|x-y|^{2}}{4t}}$, one proves that $u$ satisfies the mild
formulation;\ and it is straightforward to see that $m$ satisfies the
differential equation. Hence we have proved that $Q$ is supported by the set
of mild solutions.
\end{proof}

\begin{prop}
Assume that $u^{1},m^{1},u^{2},m^{2}$ are functions of class $C(L^{2})$, such
that $(u^{1},m^{1})$ and $(u^{2},m^{2})$ are mild solutions of the system
\ref{eq:21} corresponding to the same initial condition $(u_{0},m_{0})$, with
$u_{t},m_{t}\geq0$ for every $t\in\left[  0,T\right]  $. Then $(u^{1}%
,m^{1})=(u^{2},m^{2})$.
\end{prop}

\begin{proof}
Each $u^{i}$, $i=1,2$, satisfies the identity%
\[
u_{t}^{i}(x)=e^{tA}u_{0}+\int_{0}^{t}\nabla\cdot e^{(t-s)A}%
(u_{s}^{i}b(u_{s}^{i},F_{\zeta}\left(  m_{0}(x),\int_{0}^{s}u_{r}%
^{i}(x)dr\right)  ))ds
\]
where we have used the explicit formula for equation (\ref{micro_ecm}). This
is a closed equation and we are going to prove from it that $u^{1}=u^{2}$. A
fortiori we get also $m^{1}=m^{2}$, again from the explicit formula for
equation (\ref{micro_ecm}).

Assume by contradiction that $u^{1}\neq u^{2}$. Let $t_{0}\in\lbrack0,T)$ the
infimum of all $t\in\left[  0,T\right]  $ such that $u_{t}^{1}\neq u_{t}^{2}$.
On $\left[  0,t_{0}\right]  $ we have $(u^{1},m^{1})=(u^{2},m^{2})$. On
$\left[  t_{0},T\right]  $ we use the mild formula and property (\ref{prop:2})
to get%
\begin{align*}
\left\Vert u_{t}^{1}-u_{t}^{2}\right\Vert _{L^{2}}  &  \leq\int_{t_{0}}%
^{t}\left\Vert \nabla\cdot e^{(t-s)A}(u_{s}^{1}b(u_{s}^{1},m_{s}%
^{1})-u_{s}^{2}b(u_{s}^{2},m_{s}^{2}))\right\Vert _{L^{2}}ds\\
&  \leq\int_{t_{0}}^{t}\frac{C}{|t-s|^{1/2}}\left\Vert u_{s}^{1}b(u_{s}%
^{1},m_{s}^{1})-u_{s}^{2}b(u_{s}^{2},m_{s}^{2})\right\Vert _{L^{2}}ds\\
&  \leq\int_{t_{0}}^{t}\frac{C}{|t-s|^{1/2}}\left(  ||b(u_{s}^{1},m_{s}%
^{1})||_{\infty}\left\Vert u_{s}^{1}-u_{s}^{2}\right\Vert _{L^{2}}+||u_{s}%
^{2}||_{L^{2}}\left\Vert b(u_{s}^{1},m_{s}^{1})-b(u_{s}^{2},m_{s}%
^{2})\right\Vert _{L^{\infty}}\right)  ds.
\end{align*}
Recall that $b$ is bounded, see (\ref{bounded b}), and that $||u_{s}%
^{2}||_{L^{2}}$ is bounded by assumption. Hence%
\[
\left\Vert u_{t}^{1}-u_{t}^{2}\right\Vert _{L^{2}}\leq\int_{t_{0}}^{t}\frac
{C}{|t-s|^{1/2}}\left(  \left\Vert u_{s}^{1}-u_{s}^{2}\right\Vert _{L^{2}%
}+\left\Vert b(u_{s}^{1},m_{s}^{1})-b(u_{s}^{2},m_{s}^{2})\right\Vert
_{L^{\infty}}\right)  ds.
\]
From property (\ref{lip 1 b}) and H\"{o}lder inequality we have%
\begin{align*}
|b(u_{s}^{1},m_{s}^{1})(x)-b(u_{s}^{2},m_{s}^{2})(x)|  &  \leq C\int%
_{\mathbb{R}^{d}}e^{-|x-y|}\left(  \left\vert u_{s}^{1}\left(  y\right)
-u_{s}^{2}\left(  y\right)  \right\vert +\left\vert m_{s}^{1}\left(  y\right)
-m_{s}^{2}\left(  y\right)  \right\vert \right)  dy\\
&  \leq C\left\Vert u_{s}^{1}-u_{s}^{2}\right\Vert _{L^{2}}+C\left\Vert
m_{s}^{1}-m_{s}^{2}\right\Vert _{L^{2}}%
\end{align*}
hence%
\[
\left\Vert u_{t}^{1}-u_{t}^{2}\right\Vert _{L^{2}}\leq\int_{t_{0}}^{t}\frac
{C}{|t-s|^{1/2}}\left(  \left\Vert u_{s}^{1}-u_{s}^{2}\right\Vert _{L^{2}%
}+\left\Vert m_{s}^{1}-m_{s}^{2}\right\Vert _{L^{2}}\right)  ds.
\]
Recalling the explicit formula for equation (\ref{micro_ecm}), we have%
\begin{align*}
\left\vert m_{s}^{1}\left(  x\right)  -m_{s}^{2}\left(  x\right)  \right\vert
&  =\left\vert F_{\zeta}\left(  m_{0}(x),\int_{0}^{s}u_{r}^{1}(x)dr\right)
-F_{\zeta}\left(  m_{0}(x),\int_{0}^{s}u_{r}^{2}(x)dr\right)  \right\vert \\
&  \leq\left\Vert \partial_{b}F_{\zeta}\right\Vert _{\infty}\left\vert
\int_{0}^{s}u_{r}^{1}(x)dr-\int_{0}^{s}u_{r}^{2}(x)dr\right\vert \\
&  \leq C\int_{0}^{s}\left\vert u_{r}^{1}(x)-u_{r}^{2}(x)\right\vert
dr=C\int_{t_{0}}^{s}\left\vert u_{r}^{1}(x)-u_{r}^{2}(x)\right\vert dr
\end{align*}
whence
\[
\left\Vert m_{s}^{1}-m_{s}^{2}\right\Vert _{L^{2}}\leq C_{T}\int_{t_{0}}%
^{s}\left\Vert u_{r}^{1}-u_{r}^{2}\right\Vert _{L^{2}}dr.
\]
Summarising, the function $v_{t}:=\left\Vert u_{t}^{1}-u_{t}^{2}\right\Vert
_{L^{2}}$ satisfies
\[
v_{t}\leq\int_{t_{0}}^{t}\frac{C}{|t-s|^{1/2}}\left(  v_{s}+C_{T}\int_{t_{0}%
}^{s}v_{r}dr\right)  ds.
\]
Given $t_{1}\in\left[  t_{0},T\right]  $ we set $A\left(  t_{1}\right)  $
$:=\sup_{t\in\left[  t_{0},t_{1}\right]  }v_{t}$. Then on the interval
$t\in\left[  t_{0},t_{1}\right]  $ we have%
\[
v_{t}\leq\int_{t_{0}}^{t}\frac{C}{|t-s|^{1/2}}\left(  A\left(  t_{1}\right)
+C_{T}^{\prime}A\left(  t_{1}\right)  \right)  ds=C^{\prime\prime}A\left(
t_{1}\right)  \left(  t-t_{0}\right)  ^{1/2}.
\]
It follows that
\[
A\left(  t_{1}\right)  \leq C^{\prime\prime}A\left(  t_{1}\right)  \left(
t_{1}-t_{0}\right)  ^{1/2}.
\]

If $t_{1}-t_{0}>0$ is small enough, we deduce $A\left(  t_{1}\right)  =0$,
hence $u^{1}=u^{2}$ on $\left[  t_{0},t_{1}\right]  $, in contradiction with
the definition of $t_{0}$.
\end{proof}

\section{Simulations}

The aim of this section is to show the flexibility of the model, namely how it
may catch different kinds of aggregations. For instance, we may avoid
arbitrary concentration (even with infinitesimal noise), opposite to most of
the models in the literature; but we cover also the case of concentration,
both in the case of single and multiple concentration points. Each numerical simulation shown below is given by the following choice of parameters: number of particles $N=100$, parameter of diffusion $\sigma^2=0.1$, discretization of time $dt=10^{-4}$, Kernel smoothing parameter $\beta=0.9$ and on the initial condition we made a simple choice, choosing just a realization of uniform distribution on the square $[0,2]\times [0,2]$. 
\subsection{Degenerate aggregation}

Let us start from the most basic example, the case when each particle is
attracted by the others. Recall we model interaction between individuals and
density of population; hence each individual is pushed to high population
density regions. A standard choice for $g$ could be the following one:
\[
g(r,u)=e^{-r}\cdot\frac{u}{1+u}\quad\text{or}\quad g(r,u)=e^{-r}\cdot\tanh(u)
\]
\begin{figure}[htbp]
\centering
{\includegraphics[width=7cm]{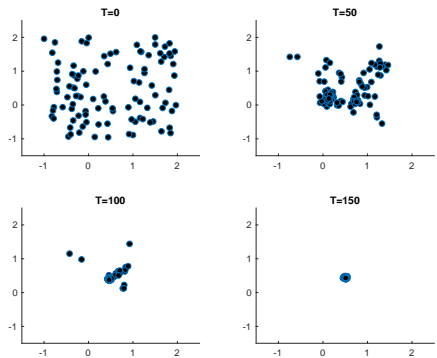}}
 \caption{Configuration of $100$ particles respectively at times $T=0,50,100,150$ with $g(r,u)=e^{-r}\cdot\frac{u}{1+u}$.}
 \label{degenerate}
\end{figure}
%

Notice that with this choice, cells continue to aggregate even at high
density. The population mass tends to concentrate into a single point (see figure $1$)

\subsection{Moderate aggregation}

Let us now include also a repulsive component in the force, to avoid collapse
of the total mass. (see figure $3$). The function $g(r,u)$ we look for should have the following features (see figure $2$):


\begin{figure}[htbp]
\centering
\includegraphics[width=7cm]{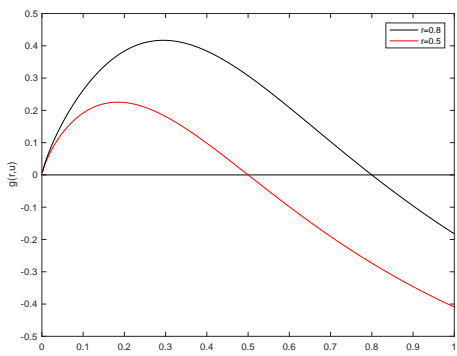}
\caption{Plot of the the function $r\mapsto g(r,u)$, for two different values of $r$. In black $r=0.8$, in red $r=0.5$}\label{intensity}
\end{figure}

\begin{itemize}
\item given the distance $r$, $g(r,u)$ is such that the force is aggregative
for small density and repulsive for huge density. This behavior is natural in
certain cases for animals: each individual is attracted by its similar, but
it does not where there are too many;

\item but there is an issue when we quantify small and huge density: this
quantification should depend on distance. At big distances, we expect that
aggregation is more relevant, and the individual tends to avoid only really
huge densities. On the contrary, at short distances, each individual is
attracted only by very small densities. 
\end{itemize}

The function we propose is the following one:
\[
g(r,u):=\frac{u\cdot\log\left(  \frac{r}{u}\right)  }{1-u\cdot\log\left(
\frac{r}{u}\right)  }%
\]

Another example could be
\[
g(r,u):=e^{-r}\cdot\frac{u\cdot(\alpha-u)}{1+u}
\]
where the parameter $\alpha$ can be interpreted as an index of overcrowding;
choosing properly $\alpha$, particles aggregate, without collapsing. The main
drawback we have observed in simulations about this alternative is its strong
sensibility to the choice of the parameter $\alpha$ with respect to the
initial configuration. The first option we propose is more stable. 

Notice that the functions $g(r,u)$ of this subsection are not product of
functions of the two single variables, namely $g(r,u)=g_{1}\left(  r\right)
g_{2}\left(  u\right)  $. 


\begin{figure}[htbp]
        \centering
        {\includegraphics[width=7cm]{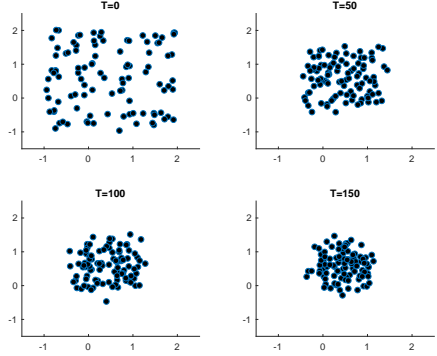}}
        \caption{Configuration of $100$ particles respectively at times $T=0,50,100,150$ with $g(r,u):=\frac{u\cdot\log\left(  \frac{r}{u}\right)  }{1-u\cdot\log\left(
\frac{r}{u}\right)  }$. }
\label{moderate2}
\end{figure}

\subsection{Aggregation in clusters}

Going back to the first model, $g(r,u)=e^{-r}\cdot\frac{u}{1+u}$, an
interesting variant is when attraction happened only up to a certain distance (see figure $4$):
\[
g_R(r,u):=\frac{u}{1+u}\cdot\exp\left(-\frac{r^2}{R}\right)
\]

With this choice we observe the formation of clusters of individuals. Clearly,
the parameter $R$ influence on the number of clusters that are generated: for
big $R$, population aggregate in a reduced number of clusters.

When $t$ goes to infinity, if the noise is infinitesimal, each cluster reduces
to a point, and maybe due to noise different clusters may meet and collapse.
\label{cluster}

\begin{figure}[h]
\centering
{\includegraphics[width=7cm]{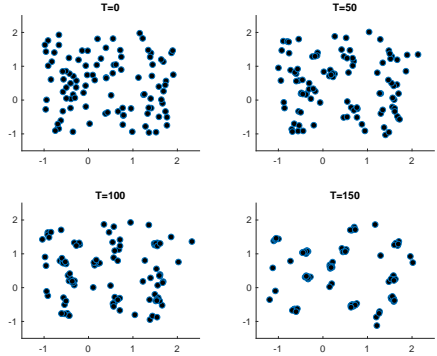}}
       \caption{Configuration of $100$ particles respectively at times $T=0,50,100,150$ with $g_R(r,u):=\frac{u}{1+u}\cdot\exp\left(-\frac{r^2}{R}\right)$ and $R=0.3$. }
        \label{cluster}
\end{figure}

\subsection{Moderate aggregation in clusters}

We may mix-up the previous two features. The following example has a tendency
to construct clusters (see figure $5$), but they remain of finite size (independently of the
noise):
\[
g_R(r,u):=\frac{u\cdot(\alpha-u)}{1+u}\cdot\exp\left(-\frac{r^2}{R}\right)
\]


\begin{figure}[h]
        \centering
          {\includegraphics[width=7cm]{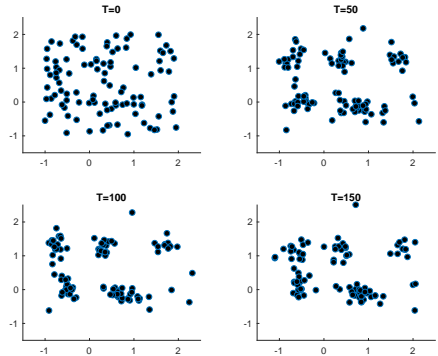}}
       \caption{Configuration of $100$ particles respectively at times $T=0,50,100,150$ with $g_R(r,u):=\frac{u\cdot(\alpha-u)}{1+u}\cdot\exp\left(-\frac{r^2}{R}\right)$ with $R=0.3$ and $\alpha=1.3$. }
        \label{cluster_moderate}
\end{figure}


\newpage
\begin{thebibliography}{99}
\footnotesize

\bibitem {ArmPainShe} {\sc Armstrong, N.J.,  Painter, K.J., Sherratt, J.A.} (2006). {A
continuum approach to modelling cell-cell adhesion.} {\em Journal of Theoretical
Biology} {\bf 243}(1), 98--113.

\bibitem{BerLof}{\sc Bergh, J. and L{\"o}fstr{\"o}m, J.} (1976). {Interpolation spaces. An introduction} {\em Springer, Berlin-New York}

\bibitem {Billy}{\sc Billingsley, P.} (2013). {Convergence of probability measures}.
John Wiley \& Sons.

\bibitem {Carrillo} {\sc Carrillo, J. A., Hittmeir, S., Volzone, B. and Yao, Y. } (2016). {Nonlinear aggregation-diffusion equations: Radial symmetry and long time
asymptotics}. {\em arXiv preprint arXiv: 1603.07767 }.

\bibitem {DapZap}{\sc Da Prato, G., Zabczyk, J.} (2014). {Stochastic equations in
infinite dimensions}. Cambridge university press.

\bibitem {Modellivari} {\sc Deroulers, C., Aubert, M., Badoual, M. and Grammaticos,  B.} (2009).
{Modeling tumor cell migration: From microscopic to macroscopic models}.
{\em Phys. Rev. E} {\bf 79}(3), 031917.

\bibitem {DysonWebb} {\sc Dyson, J., Gourley, S. A. and Webb, G. F. } (2013). {A non-local
evolution equation model of cellcell adhesion in higher dimensional space}. {\em J.
Biol. Dyn.} \textbf{7}(sup1), 68--87.

\bibitem {Evans} {\sc Evans, L.C.} (2010) {Partial Differential Equations}. {Second
Edition, Graduate Studies in Mathematics}, Vol. 19.


\bibitem {FlandoliLeimbachOlivera}{\sc Flandoli, F., Leimbach, M., Olivera,  C.} (2018). {Uniform convergence of proliferating particles to the FKPP equation.} {\em Journal of Mathematical Analysis and Applications}. 


\bibitem {Blowup} {\sc Herrero, M.A., Medina, E., 
Vel\'{a}zquez, J. J. L.} (1998). {Self-similar blow-up for a
reaction-diffusion system}. {\em Journal of Computational and Applied Mathematics}
\textbf{97}(1-2), 99--119.

\bibitem {Chemotaxis}{\sc Hillen, T. and Painter, K.J.} (2009). {A user's guide to PDE
models for chemotaxis}. {\em J. Math. Biol.} \textbf{58}(1-2), 183--217.

\bibitem{KaratzasShreve}{\sc Karatzas, I. and Shreve, S.} (1988).{ Brownian Motion and Stochastic Calculus}. {\em  Springer-Verlag, New York} 

\bibitem {KellSief}{\sc Keller, E. F. and Segel, L.A.} (1971). {Model for chemotaxis}, {\em J.
Theor. Biol.} \textbf{30}(2), 225--234.

\bibitem {KL}{\sc Kipnis, C., Landim, C.,} (2013). {Scaling limits of interacting particle systems}(Vol. 320). Springer Science $\&$ Business Media.

\bibitem {MorCapOel}{\sc Morale, D., Capasso, V., and Oelschl{\"{a}}ger, K.} (2005). {An
interacting particle system modelling aggregation behavior: from individuals
to populations}. {\em J. Math. Biol.} \textbf{50} (1), 49--66.

\bibitem {TrevisanNeklydov}{\sc Neklydov, M. and Trevisan, D.} (2015). {A particle system
approach to cell-cell adhesion models}.{\em arXiv preprint arXiv: 1601.0524}.

\bibitem {Oelsch}{\sc Oelschl{\"{a}}ger, K.} (1985). {A law of large numbers for
moderately interacting diffusion processes}.{\em Z. Wahrsch. Verw. Gebiete}
\textbf{69}(2), 279--322.

\bibitem {PainArmShe}{\sc Painter, K. J., Armstrong, N. J., Sherratt,  J. A.}  (2010). {The
impact of adhesion on cellular invasion processes in cancer and development}.
{\em J. Theor. Biol.} \textbf{264}(3), 1057--1067.

\bibitem {PainBloomShe}{\sc Painter, K. J., Bloomfield,  J. M., Sherratt, J. A., Gerisch, A.} (2015). {A nonlocal model for contact attraction and repulsion in
heterogeneous cell populations}. {\em Bull. Math. Biol.} \textbf{77}(6) , 1132--1165.

\bibitem {Perumpani}{\sc Perumpanani, A. J., Sherratt, J. A., Norbury, J., \& Byrne, H. M.} (1996). {Biological inferences from a mathematical model for malignant invasion}. {\em Invasion and Metastasis}, {\bf 16}, 209-221.

\bibitem {Popaud}{\sc Poupaud, F.} (2002). {Diagonal Defect Measures, Adhesion Dynamics
and Euler Equation}. {\em Methods Appl. Anal.} \textbf{9}(4) , 533--562

\bibitem {Simon}{\sc Simon, J.}(1987). {Compact sets in the space} $L^{p}(0,T;B)$.
{\em Ann. Mat. Pura Appl.} \textbf{146} , 65--96.


\bibitem {Sznitman}{\sc Sznitman, A.S.} (1991){Topics in propagation of chaos.} {Ecole d'Ete de Probabilites de Saint-Flour XIX| 1989,} 165-251.

\bibitem {Varadhan}{\sc Varadhan, S. R. S.} (1991). {Scaling limits for interacting
diffusions}.{\em Commun. Math. Phys.} \textbf{135}, 313--353.
\end{thebibliography}
\end{document}